\documentclass{amsart}
\usepackage[english]{babel}
\usepackage{amsfonts, amsmath, amsthm, amssymb,amscd,indentfirst}

\numberwithin{equation}{section}

\usepackage{cancel}

\usepackage{marginnote}

\usepackage{setspace}

\usepackage{url}
\usepackage{hyperref} 

\usepackage{mathrsfs}
\usepackage{euscript}

\usepackage{xcolor}
\usepackage{amsmath,amssymb,latexsym,indentfirst}
\usepackage{times}
\usepackage{palatino}

\usepackage{stmaryrd}

\usepackage{wasysym}

\usepackage{graphicx}
\graphicspath{ {./images/}}

\theoremstyle{plain}
\newtheorem{theorem}{Theorem}[section]

\newtheorem{proposition}{Proposition}[section]

\newtheorem{corollary}{Corollary}[section]

\theoremstyle{definition}
\newtheorem{definition}{Definition}[section]

\newtheorem{example}{Example}[section]

\newtheorem{remark}{Remark}[section]

\newcommand{\interior}[1]{%
	{\kern0pt#1}^{\mathrm{o}}%
}

\begin{document}

\title{Critical domains for certain Dirichlet integrals in weighted manifolds}
\author{Levi Lopes de Lima}
\address{Universidade Federal do Cear\'a (UFC),
	Departamento de Matem\'{a}tica, Campus do Pici, Av. Humberto Monte, s/n, Bloco 914, 60455-760,
	Fortaleza, CE, Brazil.}
\email{levi@mat.ufc.br}

\begin{abstract}
We start by revisiting the derivation of the variational formulae for the functional assigning to  a bounded regular domain in a fixed Riemannian manifold its first Dirichlet eigenvalue and then extend it to (not necessarily bounded) domains in certain weighted manifolds. This is further extended to other functionals defined by certain Dirichlet energy integrals, with a Morse index formula for the corresponding critical domains being established. We complement these infinitesimal results by proving a couple of {\em global} rigidity theorems for (possibly critical) domains in Gaussian half-space, including an Alexandrov-type soap bubble theorem. Although we provide direct proofs of these latter results, we find it worthwhile to point out that the main tools employed (specifically, certain Pohozhaev and Reilly identities) can be formally understood as limits (when the dimension goes to infinity) of tools previously established by Ciraolo-Vezzoni  \cite{ciraolo2017rigidity,ciraolo2019serrin} and Qiu-Xia \cite{qiu2015generalization} to handle similar problems in round hemispheres, with the notion of ``convergence'' of weighted manifolds being loosely inspired by the celebrated Poincar\'e limit theorem in the theory of Gaussian random vectors.
	\end{abstract}

\maketitle
\tableofcontents

\section{Introduction}\label{introd}

We consider a complete Riemannian manifold $(M,g)$ of dimension $n\geq 2$ and a smooth function $\phi:M\to\mathbb R$ and then form the {\em weighted manifold} $(M,g,e^{-\phi}d{\rm vol}_g)$, 
where $d{\rm vol}_g$ is the Riemannian volume element associated to $g$. 
We fix a connected, proper
domain $\Omega\subset M$ with a smooth boundary $\Sigma=\partial\Omega$. In the purely Riemannian case, where $\phi$ equals a constant, we always assume that $\overline\Omega$ is compact. Otherwise, we impose that $\Omega$ has finite $\phi$-volume (in the sense that ${\rm vol}_\phi(\Omega)<+\infty$,
where for  any Borel subset $B\subset M$ we set ${\rm vol}_\phi(B)=\int_B d{\rm vol}_\phi$, with  $d{\rm vol}_\phi=e^{-\phi}d{\rm vol}_g$).  In particular, if $M$ itself has finite $\phi$-volume, which in fact is the only case of interest here, then if needed we can assume that $d{\rm vol}_\phi$ defines a probability measure on $M$ and with this normalization our assumption on $\Omega$ boils down to $0<{\rm vol}_\phi(\Omega)<1$. 

\begin{example}\label{ex:gauss:pro}
	A notable example of a weighted manifold is the {\em Gaussian space}
		\begin{equation}\label{gauss:def}
		(\mathbb R^n,\delta, e^{-\phi_n}d{\rm vol}_\delta),\quad \phi_n(x)=\frac{|x|^2}{2}+\frac{n}{2}\log(2\pi),
	\end{equation}
	where $\delta$ is the standard flat metric, which satisfies ${\rm vol}_{\phi_n}(\mathbb R^ n)=1$. A somewhat informal interpretation of Poincar\'e limit theorem  \cite{hida1964gaussian,umemura1965infinite,mckean1973geometry,diaconis1987dozen} allows us to view this space as the appropriate limit of orthogonal projections onto $\mathbb R^n$ of (volume re-scaled) round spheres of dimension $k-1$ and radius $\sqrt{k}$ as $k\to+\infty$, a viewpoint that, at least on heuristic grounds,  we will find useful to employ here; see Remarks \ref{motiv:poin} and \ref{mod:alex} for discussions on how this heuristics applies to concrete rigidity problems and Appendix  \ref{app:poinc} for the formal results supporting it.  
	\end{example}

Our aim here is to study, initially from a variational perspective, the most elementary properties of {\em optimal} configurations for Dirichlet functionals of the type
\begin{equation}\label{dirich:funct}
	(\Omega,g,e^{-\phi}d{\rm vol}_\phi)\mapsto \mathcal E(\Omega):=\inf_{u}\int_\Omega\left(\alpha|\nabla u|^2+\beta u^2-fu\right)d{\rm vol}_\phi, 
	\end{equation}
where $\alpha,\beta\in\mathbb R$, $\alpha>0$, $f$ is an auxiliary function and $u:\Omega\to\mathbb R$ is  assumed to vary in a suitable Sobolev space of functions determined by appropriate boundary conditions (usually of Dirichlet type) and possibly satisfying an extra integral constraint. It is convenient to assume further that $\Omega$ varies within a set of domains with a {\em fixed} $\phi$-volume, a viewpoint we will always adopt here. Although a minimizer (henceforth, an {\em optimal} domain) to this variational problem is known to exist under mild conditions on the data (see for instance \cite{buttazzo2013some} and the references therein), a complete classification is far from being available in general. However, and this is our initial motivation here, we may use variational methods to get some insights on the properties an optimal domain should eventually satisfy. Instead of delving into the intricacies of a general theory, we prefer here to illustrate the methods and results by considering a couple of important examples.    

\begin{example}\label{ex:eigen} (The first Dirichlet eigenvalue)
  Let $\mathcal E(\Omega,g,d{\rm vol}_\phi)=\underline{\lambda}(\Omega,g,d{\rm vol}_\phi)$ be the functional ascribing to any weighted domain the first eigenvalue $\underline{\lambda}(\Omega,g,d{\rm vol}_\phi)$ associated to the Dirichlet problem
\begin{equation}\label{dir:prob}
	\left\{
	\begin{array}{lc}
		\Delta_\phi u+\lambda u=0 &  \Omega\\
		u=0 & \Sigma
	\end{array}
	\right.
\end{equation}
where  $\Sigma=\partial\Omega$ and
\[
\Delta_\phi={\rm div}_\phi\nabla=\Delta-\langle\nabla \phi,\nabla\,\cdot\,\rangle
\]
is the {\em weighted Laplacian}. Here, 
\begin{equation}\label{div:def}
{\rm div}_\phi=e^\phi{\rm div}_ge^{-\phi}={\rm div}\,-\langle\nabla\phi,\cdot\rangle
\end{equation}
 is the {\em weighted divergence}, $\langle\,,\rangle=g(\cdot,\cdot)$, $\Delta$ is the metric Laplacian and $\nabla$ is the metric gradient; we refer to Section \ref{prel} for a discussion on the technical assumptions needed to make sure that the existence problem for (\ref{dir:prob}) (in particular, the existence of a first eigenvalue with the expected properties) is well-posed for appropriate choices of $\Omega$. Granted these conditions, it turns out that 
\begin{equation}\label{norm:f:eigen}
	\underline{\lambda}(\Omega,g,d{\rm vol}_\phi)=\int_\Omega|\nabla u|^2d{\rm vol}_\phi, 
\end{equation}
where, besides solving (\ref{dir:prob}), $u$ is further constrained to satisfy 
\begin{equation}\label{furt:const:eigen}
	\int_\Omega u^2d{\rm vol}_\phi=1.
	\end{equation}
Due to the obvious geometric meaning of $\underline\lambda$, it is natural to investigate its variational properties under a $\phi$-volume-preserving assumption. More precisely, if $0<V<{\rm vol}_\phi(M)$,
\[
\mathscr D_V=\left\{
\Omega\subset M;{\rm vol}_\phi(\Omega)=V 
\right\}
\] 
and 
\[
\underline\lambda_V=\inf_{\Omega\in \mathscr D_V} \underline\lambda(\Omega),
\]
the question is to explictly determine those $\Omega\in \mathscr D_V$ such that $\underline\lambda_V=\underline\lambda(\Omega)$.	
\end{example}

\begin{example}\label{ex:eq:lim} (A Dirichlet energy)
For $\beta\in\mathbb R$ and $f\in L^2_\phi(\Omega)$ the {\em Dirichlet} $(\beta,f)$-{\em energy} of $\Omega\subset M$ is 
	\begin{equation}\label{energy:0}
		\mathcal E_{\beta,f}(\Omega)=\inf_{v}
		J_{\beta,f}(v),
	\end{equation}	
	where, for $v:\Omega\to\mathbb R$ vanishing on $\Sigma=\partial\Omega$,  
	\begin{equation}\label{j:func:0}
		J_{\beta,f}(v)=\int_\Omega\left(\frac{1}{2}|\nabla v|^2+\frac{\beta}{2}v^2-fv\right)d{\rm vol}_\phi
	\end{equation}
	which, under suitable conditions (cf. Remark \ref{suff:cond:inv} below), is well defined whenever $\beta>-\underline\lambda(\Omega)$ in the sense that there exists $w:\overline\Omega\to\mathbb R$ such that $\mathcal E_{\beta,f}(\Omega)=J_{\beta,f}(w)>-\infty$ with $w$ satisfying 
	\begin{equation}\label{dich:min:0}
		\left\{
		\begin{array}{lc}
			-\Delta_\phi w+\beta w=f &  \Omega\\
			w=0  & \Sigma
		\end{array}
		\right.
	\end{equation} 
	An interesting problem here is to classify the domains which minimize the energy (\ref{energy:0}) in the class of all domains with a given $\phi$-volume. Thus, 
if
\[
\mathscr E_{V,\beta,f}=\inf_{\Omega\in \mathscr D_V} \mathcal E_{\beta,f}(\Omega),
\]
the question is to explicitly determine those $\Omega\in \mathscr D_V$ such that $\mathscr E_{V,\beta,f}=\mathcal E_{\beta,f}(\Omega)$.	
	\end{example}

A recurrent theme here is to try to understand how $\mathcal E(\Omega)$ varies with $\Omega$ at an  infinitesimal level (and of course restricting ourselves to the examples above). For this we consider variations of $\Omega$ induced by one-parameter families of diffeomorphisms of the type $t\in(-\epsilon,\epsilon)\mapsto \varphi_t:M\to M$ with $\varphi_0={\rm Id}$, which we assume to be $\phi$-{\em volume-preserving} in the sense that ${\rm vol}_\phi(\Omega_t)={\rm vol}_\phi(\Omega)$ for $t\in(-\epsilon,\epsilon)$. In the context of Example \ref{ex:eigen}, and 
following the pioneering works by Garabedian and Schiffer \cite{garabedian1952convexity} and Shimakura \cite{shimakura1981premiere}, where the Euclidean case $(\mathbb R^n,\delta,d{\rm vol}_\delta)$ has been treated in detail, our first goal is to compute the variational formulae for 
\begin{equation}\label{f:s:der}
\frac{d}{dt}\underline{\lambda}(\Omega,g,d{\rm vol}_\phi)|_{t=0}\quad {\rm and} \quad 
\frac{d^2}{dt^2}\underline{\lambda}(\Omega,g,d{\rm vol}_\phi)|_{t=0}, \quad \Omega_t=\varphi_t(\Omega),
\end{equation}
in terms of the 
{\em variational function} $\xi=\langle v,\nu\rangle$, where 
$v:=d\varphi_t/dt|_{t=0}$ is the {\em variational field} and $\nu$ is the outward unit normal vector field to $\Sigma$ (Propositions \ref{first:var:p} and \ref{sec:var:en:1}). 
Needless to say, we also perform similar calculations in the setting of Example \ref{ex:eq:lim} (see the proof of Proposition \ref{opt:prop:f} culminating in (\ref{f:der:en}) and Proposition \ref{sec:var:en}).

\begin{remark}\label{hist:rem}
Our calculation for (\ref{f:s:der}) should be compared with the more general approaches in \cite{el2007domain,cavalcante2024stability} (in the purely Riemannian case) and in \cite{cunha2023hadamard} (for weighted manifolds but restricted to the first variation), where the authors handle more general classes of deformations (with the time varying metrics not being necessarily induced by embeddings of $\Omega$ into a {\em fixed} Riemannian manifold). Besides its simplicity, the advantage of the (more pedestrian) approach adopted here, which aligns with the computations in \cite{henry2005perturbation,delfour2011shapes,henrot2018shape,delay2013extremal}, lies in the fact that it can be easily extended to other naturals energy functionals, as illustrated here for the examples above in the category of weighted manifolds. 
\end{remark}

\begin{remark}\label{over:just}
A common (and certainly well-known) feature that emerges from our computations for the first variational formula is that optimal domains support functions (in our examples, the functions $u$ and $w$ satisfying (\ref{dir:prob}) and (\ref{dich:min:0}), respectively) meeting an extra boundary condition involving the constancy of its normal derivative. In other words, these functions satisfy an {\em over-determined} elliptic boundary value problem, which means that optimal domains tend to be quite rigid and hence in principle amenable to a classification; see Remark \ref{cons:crit}.  
\end{remark}

With the variational  formulae at hand, we then proceed to extend the Morse index formula first established in \cite{delima1995local} (for $\underline\lambda$ and in the Euclidean setting) to the more general framework considered here (Theorems \ref{morse:form:w} and \ref{morse:dirich}), with some interesting applications briefly indicated (Remark \ref{many:pot} and \ref{all}).  

Although variational methods certainly may be explored to access the most elementary properties of 
optimal domains for Dirichlet integrals, they do not seem to be of much use in the rather challenging problem of explicitly characterizing these domains. Thus, it is natural to seek for alternate routes by resorting to {\em global} methods relying on appropriate versions of certain classical differential and integral identities (usually associated to Pohozhaev and Reilly, respectively). Our first result in this direction is Theorem \ref{class:th}, where we classify solutions to an over-determined system associated to a certain Dirichlet energy ($\mathcal E_{-1,1}$ in the notation of Example \ref{ex:eq:lim}) defined in a domain contained in Gaussian half-space. This constitutes the exact analogue of a result in  \cite[Theorem 1.1]{ciraolo2019serrin} for domains in a round hemisphere and similarly relies on a certain Pohozhaev identity (Proposition \ref{pohoz:ident}). We stress that, as explained in Remark \ref{motiv:poin}, both our over-determined system and the associated Pohozhaev identity may be viewed as ``Poincar\'e limits'' of the corresponding  entities in \cite{ciraolo2019serrin}, which points toward the existence of an underlying, although informal, principle which may be useful in other contexts. We confirm this by means of Theorem \ref{alex:type:thm}, a version of Alexandrov's soap bubble theorem for embedded hypersurfaces in Gaussian half-space. As in the previous example, this is the exact analogue of the classical spherical case, as approached in \cite[Theorem 1.2]{qiu2015generalization}, thus similarly hinging on a Reilly-type identity (Proposition \ref{reilly:w}) which, as explained in Remark \ref{mod:alex}, may be viewed as the ``Poincar\'e limit'' of the corresponding identity in \cite[Theorem 1.1]{qiu2015generalization}. We hope that this kind of ``Poincar\'e convergence'', which is used here  
as an informal device to transfer tools and results from spheres to Gaussian space, might be successfully employed in other classes of related problems. 

\vspace{0.2cm}	
\noindent
{\bf Acknowledgments.} The author would like to thank S. Almaraz, C. Barroso and J.F. Montenegro for conversations and also to T. Shioya for comments and for making available his survey paper \cite{shioya2022metric}.

\section{Some preliminary facts}\label{prel}

We collect here the preliminary technical ingredients needed to carry out the proofs of our main results.
We insist that the {\em smooth} boundary $\Sigma=\partial\Omega$ will always be oriented by its outward pointing unit normal vector $\nu$ and is endowed with the weighted area element $d{\rm area}=e^{-\phi}d{\rm area}_g$, where $d{\rm area}_g$ is the area element induced by $g$. 
As usual, we let $L^2_\phi(\Omega)$ denote the space of all measurable functions $f:\Omega\to\mathbb R$ such that $\int_\Omega f^2d{\rm vol}_\phi<+\infty$. More generally, if $s\in\mathbb N_0=\{0,1,2,\cdots\}$ we denote by $W^s_\phi(\Omega)$ the  Sobolev space of all such functions such that its weak derivatives $\{\nabla^jf\}_{j=0}^s$ up to order $s$ lie in 
$L_\phi^2(\Omega)=W_\phi^0(\Omega)$
with the Hilbertian norm 
\begin{equation}\label{sob:norm}
\|f\|_{W^s_\phi(\Omega)}^2=\sum_{j=1}^s\|\nabla^jf\|^2_{L^2_\phi(\Omega)}. 
\end{equation}
Note that we may extend this definition to $s\in \mathbb Z$ (by duality) and to $s\in\mathbb R$ (by interpolation) Finally, if $\Omega$ is proper we denote by $W^s_{\phi[0]}(\Omega)$ the closure of $C^\infty_0(\Omega)$ in $W^s_\phi(\Omega)$. 

\begin{remark}\label{equiv:sob}
	Since $\phi$ is smooth,  whenever $\Omega$ is bounded 
	the Sobolev norm (\ref{sob:norm}) is equivalent to the standard metric Sobolev norm (which is obtained by taking $\phi\equiv 0$ and gives rise to the standard Sobolev spaces $W^s(\Omega)$, etc.). In this case we have identifications $W^s_\phi(\Omega)=W^s(\Omega)$, etc.  
	\end{remark}

We impose a few working assumptions  (which are automatically satisfied when $\Omega$ is bounded) on our otherwise general setup. 

\begin{itemize}
	\item {\bf A.1}
A {\em Poincar\'e inequality} holds for elements of $W^1_{\phi[0]}(\Omega)$: there exists $C=C_{\Omega,\phi}>0$ such that 
\begin{equation}\label{poin}
\int_\Omega |\nabla f|^2d{\rm vol}_\phi\geq C\int_\Omega f^2\,d{\rm vol}_\phi, \quad f\in 
W^1_{\phi[0]}(\Omega),
\end{equation}
which implies that the square root of the integral in the left-hand side defines a norm in $W^1_{\phi[0]}(\Omega)$ which is equivalent to $\|\,\|_{W^1_\phi(\Omega)}$. 
\item {\bf A.2}
The embedding $W^1_{\phi}(M)\hookrightarrow L^2_\phi(M)$ is compact. 
\end{itemize}

Taken together, these assumptions imply that the embedding $W^1_{\phi[0]}(\Omega)\hookrightarrow L^2_\phi(\Omega)$ is compact for any $\Omega$ as above. Thus, by standard spectral theory, the (weak) formulation of the eigenvalue problem in (\ref{dir:prob}) admits a non-trivial solution (an eigenfunction) for $\lambda$ varying in a discrete set $\{\lambda_l\}_{l=1}^{+\infty}$ with $\lambda_l\to+\infty$ as $l\to+\infty$. Also, under these conditions it is known that the first eigenvalue $\underline{\lambda}=\lambda_1>0$ is such that the corresponding eigenfunction $u$ does not change sign throughout $\Omega$, so that the corresponding eigenspace is simple (that is, one-dimensional). Hence, $\underline{\lambda}$ depends smoothly on smooth variations of $\Omega$ (for instance, variations of the type $\Omega_t=\varphi_t(\Omega)$ as in Section \ref{introd}). Finally, since both $\Sigma$ and $\phi$ are assumed to be smooth, standard elliptic theory implies that $u$ is smooth up to the boundary. Hence, if we normalize $u$ such that (\ref{furt:const:eigen}) is satisfied then $\underline\lambda$ is given by (\ref{norm:f:eigen});
we refer to the discussion in \cite[Section 2]{betta2007isoperimetric} and the references therein for details; see also Remark \ref{ex:gauss} below.

\begin{remark}\label{suff:cond:inv}
The existence of a {\em smallest} eigenvalue $\underline{\lambda}$ for $\Delta_\phi$ implies that the operator 
\[
-\Delta_\phi+\beta:W^1_{\phi[0]}(\Omega)\to L^2_\phi(\Omega)
\] is invertible if $\beta>-\underline\lambda$, which justifies the existence of a minimizer for $J_{\beta,f}$ in (\ref{j:func:0}) via a solution of (\ref{dich:min:0}).
\end{remark}

The assumptions above suffice to justify the computations leading to the variational formulae in Sections \ref{var:form} and \ref{dich:energy}. However, in order to establish finer properties of the underlying functional (such as the Morse index formula in Theorems \ref{morse:form:w} and \ref{morse:dirich}) one more requirement is needed.  

\begin{itemize}
	\item {\bf A.3} The standard ``elliptic package'' for non-homogeneous boundary problems of the type
	 \begin{equation}\label{dir:prob:non}
	 	\left\{
	 	\begin{array}{lc}
	 		\Delta_\phi u=f &  \Omega\\
	 		u=\xi & \Sigma
	 	\end{array}
	 	\right.
	 \end{equation}
	 holds true, including suitable trace/extensions theorems, the fact that the map
	 \begin{equation}\label{iso:ell}
	 u\in W^{s+2}_\phi(\Omega)\mapsto (\Delta_\phi u,u|_\Sigma)\in W_\phi^{s}(\Omega)\times W_\phi^{s+3/2}(\Sigma), \quad s\geq 0,
	 \end{equation}
	 is an isomorphism and the corresponding regularity theory (up to the boundary); see Remark \ref{package}.
\end{itemize}

\begin{remark}\label{ex:gauss}
	Poincar\'e inequality (\ref{poin}) holds in the {Gaussian} setting of Example \ref{ex:gauss:pro} 
	with $\Omega$ being any proper domain with ${\rm vol}_{\phi_n}\Omega)<1$. 
	Proofs may be found in \cite{ehrhard1984inegalites,di2002linear,chiacchio2004comparison} and the method of proof, which is based on Gaussian rearrangements, may be extended to weighted manifolds for which the validity of a certain relative isoperimetric inequality is taken for granted \cite{talenti1997weighted}.
	On the other hand,  the compactness of the embedding $W^1_{\phi}(M)\hookrightarrow L^2_\phi(M)$ seems to hold in a more general setting (if we further assume that ${\rm vol}_\phi(M)<+\infty$). 
	For instance, it holds 
	for $(\mathbb R^n,\delta,e^{-\phi}d{\rm vol}_\delta)$, with $\phi(x)\sim |x|^\theta$, $\theta>1$, as $|x|\to+\infty$ \cite{hooton1981compact}. More generally (that is, in the presence of a Riemannian background $(M,g)$) it is well-known that this property follows from the validity of a certain logarithmic Sobolev inequality, which by its turn holds true whenever the corresponding weighted manifold satisfies ${\rm Ric}_\phi\geq \rho g$, $\rho>0$,  where 
	\begin{equation}\label{bakry:ricci}
		{\rm Ric}_\phi={\rm Ric}_g+\nabla^2\phi
		\end{equation}
		 is the Bakry-\'Emery Ricci tensor; see \cite[Section 5.7]{bakry2014analysis}, where this is discussed in the abstract framework of diffusion Markov triples satisfying the curvature-dimension condition $CD(\rho,\infty)$, $\rho>0$. 
	\end{remark}

\begin{remark}\label{package}
	If $\overline\Omega$ is compact then the elliptic package mentioned above is classically known \cite{lions2012non,folland2020introduction}. In the case $(\mathbb R^n,\delta, e^{-\phi}d{\rm vol}_\delta)$, it is established in \cite{harrington2014sobolev} under the assumption that $\Sigma$ is smooth with a positive reach and under mild assumptions on $\phi$ which are too complicated to exactly reproduce here but which essentially boil down to assuming that $\phi(x)$ grows at least as $|x|^2$ as $|x|\to +\infty$). In particular, this covers the Gaussian case in (\ref{gauss:def}). Keeping the same growth conditions on $\phi$, we further mention that the methods in \cite{harrington2014sobolev} may be adapted without much difficulty to the general case $(M,g,e^{-\phi}d{\rm vol}_g)$ if there exist a compact $K\subset M$, $r>0$ and a diffeomorphism $\psi:(M\backslash K,g)\to(\mathbb R^n\backslash B_r(\vec{0}),\delta)$ with uniformly bounded distortion (up to a sufficiently high order).  
	\end{remark}

\begin{remark}\label{AAA:123}
The assumptions {\bf A.1}, {\bf A.2} and {\bf A.3} above entail the fact that both $\Omega$ and $\Sigma$ are sufficiently well behaved at infinity (both topologically and metrically) in order that the validity of a couple of technical ingredients further needed below is ensured. 
First, the following integration by parts formula holds:
\begin{equation}\label{int:part}
	\int_\Omega f\Delta_\phi h\, d{\rm vol}_\phi+\int_\Omega\langle\nabla f,\nabla h\rangle\, d{\rm vol}_\phi=
	\int_\Sigma f\frac{\partial h}{\partial\nu}d{\rm area}_\phi, 
	\end{equation}
	where $f\in W_\phi^1(\Omega)\cap C^\infty(\Omega)$ and $h\in W_\phi^2(\Omega)\cap C^\infty(\Omega)$; this follows from the corresponding divergence theorem
	\begin{equation}\label{div:thn:wei}
		\int_\Omega {\rm div}_\phi Y d{\rm vol}_\phi=\int_\Sigma \langle Y,\nu\rangle d{\rm area}_\phi,
		\end{equation}
where now $Y$ is vector field on $\overline \Omega$ lying in a suitably defined Sobolev space. 
We remark that this latter formula follows 	
	from the standard one (for bounded domains) by a simple approximation.
	In the same vein, $\Sigma$ is is assumed to support its own weighted Sobolev scale $W_\phi^{s}(\Sigma)$, so that 
		the intrinsic analogue of (\ref{int:part}) holds:
		\begin{equation}\label{int:part:s}
			\int_\Sigma f\Delta_{\Sigma,\phi} h\, d{\rm area}_\phi=-\int_\Sigma\langle D f, D h\rangle\, d{\rm area}_\phi,	
			\end{equation}
			where now $f$ and $h$ are functions on $\Sigma$ lying in a suitable Sobolev space (say, $W^2_\phi(\Sigma)$),  $D$ is the covariant derivative of $g|_\Sigma$ and we use self-explanatory notation stemming from the fact that $\Sigma$ itself is a weighted manifold with the induced structures. As before, this should follow from the analogue of (\ref{div:thn:wei}),
		\begin{equation}\label{div:thm:wei:s}
			\int_\Sigma {\rm div}_{\Sigma,\phi} Y d{\rm area}_\phi=0,
			\end{equation}
			where $Y$ is tangent to $\Sigma$ (and also lies in a suitable Sobolev space). Finally, we note that at some points of this paper, specifically in the proofs of Theorem \ref{class:th} and \ref{alex:type:thm} (which concern properties of domains in Gaussian space) we assume that $\Sigma$ is such that certain test functions, which depend polynomially on $|x|$, lie in $W^2_\phi(\Sigma)$, so that the integration by parts formulae above may be applied. In this regard, polynomial volume growth at infinity is more than enough.  
\end{remark}
	
\noindent
{\bf Standing assumptions:} Henceforth we will always take it for granted that the weighted domain $(\Omega,g,e^{-\phi}d{\rm vol}_g)$ and its weighted boundary $(\Sigma,g|_\Sigma,e^{-\phi}d{\rm area}_g)$ meet the assumptions {\bf A.1}, {\bf A.2} and {\bf A.3} above; see also Remarks \ref{suff:cond:inv}, \ref{ex:gauss}, \ref{package} and \ref{AAA:123}. 	
	
\vspace{0.3cm}	
	
Another ingredient we will use is 
a simple adaptation to the weighted setting of certain variational formulae for bulk and surface integrals which are well-known in the area of shape optimization \cite{henry2005perturbation,delfour2011shapes,henrot2018shape}. To state it we recall that the  {\em weighted mean curvature} of $\Sigma$ is 
\begin{equation}\label{mean:curv:weig}
	H_\phi={\rm div}_\phi\nu=H-\langle \nabla \phi,\nu\rangle,
	\end{equation} 	
	where $H={\rm div}\,\nu$ is the usual mean curvature. 

\begin{proposition}\label{main:formulae}
	For any variation $t\in (-\epsilon, \epsilon)\mapsto\Omega_t$ as above (not necessarily preserving volume) and any smooth family of smooth maps $t\in (-\epsilon, \epsilon)\mapsto \Psi(t, \cdot): \Omega_t\to \mathbb R$ such that
	$\Psi(t,\cdot)$ and $\partial\Psi/\partial t(t,\cdot)$ lie in $W^1_\phi(\Omega_t)\cap C^\infty(\Omega_t)$  there hold
	\begin{equation}\label{main:formulae:b}
		\frac{d}{dt}\int_{\Omega_t}\Psi\, d{\rm vol}_\phi=\int_{\Omega_t}\frac{\partial\Psi}{\partial t}d{\rm vol}_\phi+
		\int_{\Sigma_t}\Psi\xi \,d{\rm area}_\phi, \quad\Sigma_t=\partial\Omega_t,
		\end{equation}
		and
		\begin{equation}\label{main:formulae:s}
				\frac{d}{dt}\int_{\Sigma_t}\Psi\, d{\rm area}_\phi=\int_{\Sigma_t}\left(
				\frac{\partial\Psi}{\partial t}+\left(\langle\nabla\Psi,\nu\rangle+H_\phi\Psi\right)\xi
				\right)d{\rm area}_\phi,
			\end{equation}
			where $\nabla$ here means the metric gradient with respect to the $x$ variable. 
	\end{proposition}
	
	\begin{proof}
		The simple proof \cite[Theorem 1.11]{henry2005perturbation} in the Euclidean case (with $\Omega$ bounded) applies {\em verbatim} to the Riemannian setting. From this the weighted versions above (again with $\Omega$ bounded) may be easily derived by straightforward manipulations. The general case, in which $\Omega$ is only assume to have  a finite $\phi$-volume, follows by a simple approximation (based on our standing assumptions).  
	\end{proof}
	
	\begin{remark}\label{app:var:gen}
	If $\Psi=1$ in (\ref{main:formulae:b}) we obtain the first variation formula for the $\phi$-volume:
	\[
	\frac{d}{dt}_{|_{t=0}}{\rm vol}_\phi(\Omega_t)=\int_\Sigma \xi \,d{\rm area}_\phi.
	\]
	In particular, if the variation is $\phi$-volume-preserving then the variational function $\xi$ satisfies  
	\begin{equation}\label{cond:mean}
	\int_\Sigma \xi \,d{\rm area}_\phi=0.
	\end{equation}
	Conversely, it is well-known that for any compactly supported, smooth function $\xi:\Sigma\to\mathbb R$ satisfying (\ref{cond:mean}) there exists a $\phi$-volume-preserving variation whose variational vector is $\xi\nu$. For any such variation we may use (\ref{main:formulae:s}) with $\Psi=\xi$ to check that 
	\begin{equation}
		\label{cond:mean:2}
		0=	\frac{d^2}{dt^2}_{|_{t=0}}{\rm vol}_\phi(\Omega_t)=
			\int_\Sigma\left(\dot\xi+\left(\frac{\partial\xi}{\partial\nu}+H_\phi\xi\right)\xi\right)d{\rm area}_\phi,
	\end{equation} 
	where here and in the following a dot means $d/dt|_{t=0}$. 
	\end{remark}

\section{The variational formulae for $\underline\lambda$}\label{var:form}	
 
 We start by stating our version of the classical {\em Hadamard formula} for the first variation of $\underline{\lambda}$; see also \cite{cunha2023hadamard} and the references therein for other versions of this result in the weighted setting.
 
 \begin{proposition}\label{first:var:p}
 If 
 $\Omega_t=\varphi_t(\Omega)$ is a (not necessarily $\phi$-volume-preserving) variation of domains as above and $\underline{\lambda}_t$ is the corresponding first eigenvalue then
 \begin{equation}\label{first:var:f}
 	\underline{\dot\lambda}=-\int_\Sigma \left(\frac{\partial u}{\partial \nu}\right)^2\xi\,d{\rm area}_\phi.
 \end{equation}
 \end{proposition}
 
\begin{proof} 
	The corresponding first eigenfunction $u_t$ satisfies
	\begin{equation}\label{dir:prob:t}
		\left\{
		\begin{array}{lc}
			\Delta_\phi u_t+\underline{\lambda}_t u_t=0 &  \Omega_t\\
			u_t=0 & \Sigma_t
		\end{array}
		\right.
	\end{equation}
so upon derivation we see that $\dot u$ satisfies 
  \begin{equation}\label{dir:prob:t:dt}
 	\left\{
 	\begin{array}{lc}
 		\Delta_\phi\dot u+\dot{\underline{\lambda}} u+\underline{\lambda} \dot u=0 &  \Omega\\
 		\dot u=-\frac{\partial u}{\partial\nu}\xi & \Sigma
 	\end{array}
 	\right.
 \end{equation}
We now observe that, due to our normalization (\ref{furt:const:eigen}),
\[
1=\int_{\Omega_t}u_t^2d{\rm vol}_\phi, 
\]
so if we use (\ref{main:formulae:b}) with $\Psi=u_t^2$ we get 
\[
0=2\int_\Omega u\dot u \,d{\rm vol}_\phi+\int_\Sigma u^2\xi \,d{\rm area}_\phi, 
\]
 and since $u=0$ on $\Sigma$,
 \begin{equation}\label{orth:c}
 	\int_\Omega u\dot u\,d{\rm vol}_\phi=0. 
 \end{equation}
We also have, this time using the normalization (\ref{norm:f:eigen}),  
\[
\underline{\lambda}_t=\int_{\Omega_t}|\nabla u_t|^2d{\rm vol}_\phi,   
\]
which by (\ref{main:formulae:b}) with $\Psi=|\nabla u_t|^2$ gives   
\begin{equation}\label{lamd:dot:prel}
\underline{\dot\lambda}
 = 2
{\int_\Omega\langle\nabla u,\nabla\dot u\rangle d{\rm vol}_\phi}+
\int_\Sigma|\nabla u|^2\xi \,d{\rm area}_\phi.
\end{equation}
By means of (\ref{int:part}) we can handle the first integral in the right-hand side above in two slightly different ways. We first have 
\begin{eqnarray*}
{\int_\Omega\langle\nabla u,\nabla\dot u\rangle d{\rm vol}_\phi}
& = & 
 -\int_\Omega u\Delta_\phi \dot u \,d{\rm vol}_\phi+\int_\Sigma u\frac{\partial\dot u}{\partial\nu}d{\rm area}_\phi \\
& \stackrel{(\ref{dir:prob:t:dt})}{=} &  
\int_\Omega u\left(\underline{\dot\lambda} u+\lambda\dot u\right) d{\rm vol}_\phi,
\end{eqnarray*}
so that (\ref{orth:c}) gives
\begin{equation}\label{lamb:1}
{\int_\Omega\langle\nabla u,\nabla\dot u\rangle d{\rm vol}_\phi}=\underline{\dot\lambda}.	
	\end{equation}
On the other hand, 	
\begin{eqnarray*}
	{\int_\Omega\langle\nabla u,\nabla\dot u\rangle d{\rm vol}_\phi}
	& = & 
	-\int_\Omega \dot u\Delta_\phi u \,d{\rm vol}_\phi+\int_\Sigma \dot u\frac{\partial u}{\partial\nu}\,d{\rm area}_\phi \\
	& \stackrel{(\ref{dir:prob})+(\ref{dir:prob:t:dt})}{=} &  
	-\int_\Omega \dot u(-\lambda u) d{\rm vol}_\phi+
	\int_\Sigma \left(-\frac{\partial u}{\partial \nu}\xi\right)\frac{\partial u}{\partial \nu}d{\rm area}_\phi,
\end{eqnarray*}
so that, again using (\ref{orth:c}),  
\begin{equation}\label{lamb:2}
	{\int_\Omega\langle\nabla u,\nabla\dot u\rangle d{\rm vol}_\phi}=-\int_\Sigma \left(\frac{\partial u}{\partial \nu}\right)^2\xi\, d{\rm area}_\phi,	
\end{equation}
and the result follows from this and 
(\ref{lamb:1}).
\end{proof}

We now explore (\ref{first:var:f}) by recalling a definition first put forward in \cite{shimakura1981premiere}.
	
	\begin{definition}\label{crit:shima}
	Let $\Sigma'\subset\Sigma$ be an open subset. We say that $(\Omega,g)$ is $\Sigma'$-{\em critical} (for $\underline\lambda$)  if $\underline{\dot\lambda}=0$ for any $\phi$-volume-preserving deformation such that ${\rm supp}\,\xi\subset\Sigma'$.
	\end{definition}
	
	
	\begin{remark}\label{cons:crit}
		That $\Omega$ is $\Sigma'$-critical is equivalent to the first eigenfunction $u$ satisfying
	\begin{equation}\label{crit:cond}
		\left\{
		\begin{array}{lc}
			\Delta_\phi u+\underline{\lambda} u= 0 & \Omega\\
			s_\xi= \xi & \Sigma\\
		|\nabla u|=\kappa & \Sigma'	
		\end{array}
		\right.
	\end{equation}	
	for some constant $\kappa>0$. In particular, if $\Sigma'=\Sigma$ and $\Omega$ lies in a simply connected space form $\mathscr S^n_K$ with sectional curvature $K\in \mathbb R$, results in \cite{serrin1971symmetry,kumaresan1998serrin} imply that $\Omega$ is a geodesic ball (where $\Omega$ is required to lie in a hemisphere in the spherical case $K>0$). On the other hand, it is clear that any annular domain $\Omega\subset \mathscr S^n_K$ enclosed by geodesic spheres (so that $\partial\Omega=\mathbb S^{n-1}_{r_1}(x_1)\cup\mathbb S^{n-1}_{r_2}(x_1)$, the $r_i>0$ and $x_i\in  \mathscr S^n_K$, $i=1,2$) is $\mathbb S^{n-1}_{r_i}(x_i)$-critical for $i=1,2$. Thus, the localization in Definition \ref{crit:shima} allows for some more flexibility in the somewhat arduous task of finding non-trivial, explicit solutions to the over-determined system (\ref{crit:cond}).
		\end{remark}
		
	\begin{remark}\label{opt:cond:l}
		If $\Omega$ is optimal for $\underline\lambda$ then it is $\Sigma$-critical and (\ref{crit:cond}) means that the ``optimal'' eigenfunction $u$ satisfies an extra condition along $\Sigma$ (its normal derivative is constant). Of course, this aligns with the comments in Remark \ref{over:just}. 
		\end{remark}	

	We now proceed to compute $\underline{\ddot\lambda}$ assuming that $(\Omega,g)$ is $\Sigma'$-critical. 
	We first notice that under $\Sigma'$-criticality any such $\xi$ admits a natural extension towards $\Omega$. Indeed, observing that $\dot u=0$ outside ${\rm supp}\,\xi$, it follows from (\ref{dir:prob:t:dt}) that $s_\xi:=-(\partial u/\partial \nu)^{-1}\dot u$ satisfies
\begin{equation}\label{inter:step:2}
	\left\{
	\begin{array}{lc}
		\Delta_\phi s_\xi+\underline{\lambda} s_\xi= 0 & \Omega\\
		s_\xi= \xi & \Sigma\\
		\int_\Omega s_\xi u\,d{\rm vol}_\phi=0
	\end{array}
	\right.
\end{equation}  
In particular, this allows us to define $\partial\xi/\partial\nu:=\partial s_\xi/\partial \nu$ along $\Sigma$. 
We may now state our version of the second variation formula for $\underline\lambda$ in the weighted setting; see Remark \ref{hist:rem} for other approaches to this calculation in the purely Riemannian setting. 

\begin{proposition}\label{sec:var:f:w:l}
	If $\Omega$ is $\Sigma'$-critical (for $\underline\lambda$) then 
\begin{equation}\label{shim:q}
	\underline{\ddot\lambda} = 2\kappa^2Q_\phi(\xi),
\end{equation}
where the quadratic form $Q_\phi$ is given by 
\begin{equation}\label{shim:qq}
	Q_\phi(\xi)=\int_{\Sigma'} \xi \mathcal L_\phi\xi\, d{\rm area}_\phi, \quad \mathcal L_\phi\xi=\frac{\partial \xi}{\partial\nu}+H_\phi\xi. 
\end{equation}	
	\end{proposition}

\begin{proof}
An useful information here is obtained by noticing that, along $\Sigma'$, where $|\nabla u|=|\partial u/\partial \nu|=\kappa$,
\[
0  =  -\underline{\lambda} u =\Delta_\phi u=\Delta_gu-\frac{\partial u}{\partial \nu}\langle\nabla \phi,\nu\rangle,
\]
where $\nabla^2$ is the metric Hessian operator.
Since
\[
\Delta_gu=\Delta_{\Sigma} u+\frac{\partial u}{\partial \nu} H+(\nabla^2u)(\nu,\nu)=\frac{\partial u}{\partial \nu} H+(\nabla^2u)(\nu,\nu)
\]
and $\Delta_{\Sigma} u=0$, where $\Delta_{\Sigma}$ is the intrinsic Laplacian of $\Sigma$, we conclude that 
\begin{equation}\label{H:comp}
	(\nabla^2u)(\nu,\nu)=-\frac{\partial u}{\partial \nu}H_\phi. 
	\end{equation}
With these preliminaries at hand we  note that, by (\ref{first:var:f}),
\[
\frac{d\underline{\lambda}}{dt}=-\int_{\Sigma'_t}\langle\nabla u_t,\nu_t\rangle^2\xi_t\,d{\rm area}_{\phi},  \quad \Sigma'_t=\varphi(\Sigma'), 
\]
so that (\ref{main:formulae:s}) gives 
\begin{equation}\label{var:gen}
-\underline{\ddot\lambda}=\int_{\Sigma'}\left(\dot\Psi+\left(\langle \nabla \Psi,\nu\rangle+\Psi H_\phi\right)\xi \right)d{\rm area}_{\phi}, \quad \Psi=\langle\nabla u_t,\nu_t\rangle^2\xi_t.	
	\end{equation}
	We now compute the various terms in the right-hand side above. First, 
	\[
		\dot\Psi  = 
		2 \langle \nabla u,\nu\rangle\left(\langle \nabla \dot u,\nu\rangle+\langle \nabla u ,\dot\nu\rangle\right)\xi
		+\langle \nabla u,\nu\rangle^2\dot\xi.
		\]
		Now recall that $\dot u=-\partial u/\partial \nu\, s_\xi$ and 
 $|\nabla u|=\kappa$ along $\Sigma'$. Also, without loss of generality we may assume that $v$ is normal ($v=\xi\nu$), which gives  $\dot\nu=-D\xi$ \cite[Proposition 2.2]{zhu2002lectures}; recall that $D$ is the intrinsic covariant derivative of $\Sigma$.
Hence,
\[
\dot\Psi=-2\kappa^2\xi\frac{\partial \xi}{\partial\nu}+\kappa^2\dot\xi. 
\]	
On the other hand, again using $\Sigma'$-criticality,
\[
\langle \nabla \Psi,\nu\rangle+\Psi H_\phi=\nu\left(\langle \nabla u,\nu\rangle^2\right)\xi+
\kappa^2\left(\frac{\partial\xi}{\partial\nu}+H_\phi\xi\right).
\]
But
\[
\nu\left(\langle \nabla u,\nu\rangle^2\right)=2\frac{\partial u}{\partial \nu}\left(\langle\nabla_\nu\nabla u,\nu\rangle+\langle \nabla u,\nabla_\nu\nu\rangle\right)
\]
and since 
\[
\langle \nabla u,\nabla_\nu\nu\rangle=\frac{\partial u}{\partial \nu}\langle\nu,\nabla_\nu\nu\rangle=0
\]
we find that 
\[
\nu\left(\langle \nabla u,\nu\rangle^2\right)=2\frac{\partial u}{\partial \nu}\langle\nabla_\nu\nabla u,\nu\rangle
\stackrel{(\ref{H:comp})}{=}-2\kappa^2H_\phi.
\]
Thus, if we put together the pieces of our computation we get
\[
\underline{\ddot\lambda} =2\kappa^2\int_{\Sigma'} \xi\left(\frac{\partial \xi}{\partial\nu}+H_\phi\xi\right)d{\rm area}_{\phi} -\kappa^2\int_\Sigma\left(\dot\xi+\left(\frac{\partial\xi}{\partial\nu}+H_\phi\xi\right)\xi\right)d{\rm area}_{\phi},
\]
so the proof is completed using (\ref{cond:mean:2}).
\end{proof}

\section{A Morse index formula}\label{morse}

The computations leading to (\ref{shim:q})-(\ref{shim:qq}) justify the following definitions, which are direct extensions of concepts already appearing in \cite{shimakura1981premiere,delima1995local}. Given $\widetilde\Sigma\subset\Sigma$, we denote by $\widehat{W}^1_{\phi[0]}(\widetilde\Sigma)$ the space of functions $\xi\in {W}^1_{\phi[0]}(\widetilde\Sigma)$ such that $\int_\Sigma\xi d{\rm area}_\phi=0$.

\begin{definition}\label{stab}
	If $\Omega\subset M$ is $\Sigma'$-critical (for $\underline\lambda$) with $\Sigma'\subset\Sigma$, then we say that it is  $\Sigma'$-{\em stable} if $Q_\phi(\xi)\geq 0$ for any $\xi\in \widehat{W}^{1/2}_{\phi[0]}(\Sigma')$. In case $Q_\phi|_{\widehat{W}^{1/2}_{\phi[0]}(\Sigma')}$ is positive definite then we say that $\Omega$ is {\em strictly} $\Sigma'$-{\em stable}. If $\Sigma'=\Sigma$ then we simply say that $\Omega$ is stable (or strictly stable). 
\end{definition} 

\begin{definition}\label{index}
If $\Omega$ is $\Sigma'$-critical (with $\Sigma'\subset \Sigma$) and $\Sigma''\subset\Sigma'$, we define  the {\em nullity} and the {\em index}  of the pair $(\Omega,\Sigma'')$ by 
\[
{\rm null}(\Omega,\Sigma'')=\dim\left\{\xi\in \widehat{W}^{1/2}_{\phi[0]}(\Sigma'')); Q_\phi(\xi)=0 \right\},
\]	
and 
\[
{\rm ind}(\Omega,\Sigma'')=\sup \left\{\dim V; V\subset \widehat{W}^{1/2}_{\phi[0]}(\Sigma'') \, {\rm and }\,Q_\phi|_V \,{\rm is}\, {\rm positive}\, {\rm definite}
\right\},
\]
respectively. 
\end{definition}

\begin{remark}\label{shim:r}
	It  has been proved in \cite{shimakura1981premiere} (in the classical Euclidean setting) that any $\Sigma'$-critical domain is {\em locally} strictly stable in the sense that for any $x_0\in\Sigma'$ there exists a small neighborhood $U\subset\Sigma'$ with $x_0\in U$ such that $\Omega$ is strictly $U$-stable. In \cite{delima1995local} this result has been reformulated in a more conceptual framework (again in the classical case) by means of the establishment of a Morse index formula for the quadratic form $Q_\phi$. We will check below that this latter result may be extended to our more general setting  with essentially the same proof (Theorem \ref{morse:form:w}).
\end{remark}

\begin{remark}\label{ref:cavnunes}
	The stability (or lack thereof) of (purely) Riemannian domains has been recently studied from a {\em global} viewpoint (i.e with $\Sigma'=\Sigma$) in \cite{cavalcante2024stability}, notably in the two-dimensional case ($n=2$).  In this regard, 
	it follows from Faber-Krahn inequality \cite{chavel1984eigenvalues} that a geodesic ball $B_r(x)\subset \mathscr S^n_K$ in a simply connected space form 
	$\mathscr S^n_K$ with sectional curvature $K\in \mathbb R$
	is the only (volume-preserving) minimizer for $\underline\lambda$. In particular, this geodesic ball is not only $\mathbb S^{n-1}_r(x)$-critical but also strictly $\mathbb S^{n-1}_r(x)$-stable, where $\mathbb S^{n-1}_r(x)=\partial B_r(x)$ is the associated geodesic sphere. An interesting problem, first put forward in  \cite{cavalcante2024stability}, is to check whether in the spherical case $K>0$ the converse of this statement (namely, if $\Omega\subset \mathscr S_K^n$ is $\Sigma$-stable then $\Omega$ is a geodesic ball) holds true. Despite an affirmative answer in \cite[Theorem 1.1]{cavalcante2024stability} for $n=2$, this problem remains wide open if $n\geq 3$, except for a further contribution in \cite[Theorem 1.4]{cavalcante2024stability}, where it is shown that the domain is a hemisphere in case its boundary is minimal. 
	\end{remark}
	
	\begin{remark}\label{anal:fb:gauss}
		A Faber-Krahn inequality also holds in the Gaussian setting, with the optimal domains being the Gaussian half-spaces
		\begin{equation}\label{gauss:dir:g}
			\mathcal H^n_{{\bf u},d}=\left\{
			x\in\mathbb R^n;\langle x,{\bf u}\rangle_\delta>d
			\right\}. 
		\end{equation}
		where ${\bf u}\in\mathbb R^n$, $\|{\bf u}\|=1$, and $d\in\mathbb R$  \cite[Theorem 3.1]{betta2007isoperimetric}. In particular, one may ask whether these are the only stable domains for $\underline\lambda$ in Gaussian space. In fact, Poincar\'e limit theorem (cf. Appendix \ref{app:poinc}) suggests that this should hold true if and only if the corresponding conjectured property for $\mathscr S^n_K$ mentioned in Remark \ref{ref:cavnunes} holds as well (at least for all $n$ large enough). 
		\end{remark}	

We now turn to the numerical invariants appearing in Definition \ref{index}. At first sight, it is not clear how the index/nullity relate to an algebraic count of negative/null ``eigenvalues'' of the quadratic form $Q_\phi$, which is defined in terms of the manifestly {\em non-local} operator $\mathcal L_\phi$. In particular, it is not even clear that these invariants are finite if $\overline{\Sigma'}$ is compact. However, as already mentioned in Remark \ref{shim:r}, standard results in elliptic theory ensured by our standing assumption (specially {\bf A.3}) may be used to settle this problem. A first ingredient here is the following G\"arding inequality satisfied by $Q_\phi$. 

\begin{proposition}\label{garding}
	There exist constants $C_1,C_2>0$ (depending only on $\Omega$) such that, 	for any $\xi\in W^{1/2}(\Sigma)$, 
	\begin{equation}\label{garding:1}
		Q_\phi(\xi)\geq C_1\|\xi\|^2_{W^{1/2}(\Sigma)}-C_2\|\xi\|^2_{W^{0}(\Sigma)},
		\end{equation}
where $W^{s}(\Sigma)$, $s\in\mathbb R$, is the standard Sobolev scale of $\Sigma$. 
	\end{proposition}  
	
	\begin{proof}
	We follow the proof of \cite[Theorem 4.2]{shimakura1981premiere} closely (see also \cite{delima1995local} for a slightly different argument). Let us 
	consider the isomorphism 
		\begin{equation}\label{isom:he}
		\xi\in W^{s}_\phi(\Sigma)\mapsto h_\xi\in W_\phi^{s+1/2}(\Omega)  
	\end{equation}
	that to each $\xi$ associates its $\phi$-harmonic extension (that is, $\Delta_\phi h_\xi=0$ on $\Omega$); cf. (\ref{iso:ell}). 
	Also,
	let us set $w_\xi=s_\xi-h_\xi$, where $s_\xi\in W_\phi^{s+1/2}(\Omega)$ is given by (\ref{inter:step:2}). Since $\Delta_\phi w_\xi=-\underline\lambda s_\xi$, elliptic regularity implies that $w_\xi\in  W_\phi^{s+5/2}(\Omega)$. Hence, standard trace theory gives that $\partial \xi/\partial\nu=\partial w_\xi/\partial\nu\in W_\phi^{s+1}(\Gamma)$, so that the operator $\xi\in W_\phi^{s}(\Sigma)\mapsto \mathcal L_\phi\xi\in W_\phi^{s+1}(\Sigma)$ has order $-1$. Thus, if 
	\begin{equation}\label{def:q2}
	\mathscr Q_\phi(\xi)=\int_\Sigma \xi\frac{\partial\mathfrak s_\xi}{\partial\nu}d{\rm vol}_\phi
	\end{equation}
	for $\xi\in W_\phi^{1/2}(\Sigma)\subset W_\phi^{0}(\Sigma)$,
	we may apply this with $s=0$ to obtain
	\[
	\left|Q_\phi(\xi)-\mathscr Q_\phi(\xi)\right|\leq C_3\|f\|^2_{W_\phi^0(\Sigma)}, \quad C_2>0,
	\]
	from which we see that 
	\[
	Q_\phi(\xi)\geq \mathscr Q_\phi(\xi)-C_3\|\xi\|^2_{W^0(\Sigma)}. 
	\]
	On the other hand, by (\ref{def:q2}) and (\ref{int:part}),
	\begin{eqnarray*}
		\mathscr Q_\phi(\xi)
		& = & \int_\Omega|\nabla h_\xi|^2 d{\rm vol}_\phi\\
		& = & \|h_\xi\|^2_{W^1_\phi(\Omega)}- \|h_\xi\|^2_{W^0_\phi(\Omega)}\\
		& \geq 	& C_1\|\xi\|^2_{W_\phi^{1/2}(\Sigma)} - C_4\|\xi\|^2_{W_\phi^{-1/2}(\Sigma)},
	\end{eqnarray*}
	where $C_1, C_4>0$ and we used (\ref{isom:he}) with $s=1/2$ and $s=0$ in the last step. Now, the Sobolev embedding $W_\phi^{0}(\Sigma)\hookrightarrow W_\phi^{-1/2}(\Sigma)$ gives 
	\[
	\|\xi\|_{W_\phi^{-1/2}(\Sigma)}\leq C_5\|\xi\|_{W_\phi^{0}(\Sigma)}, \quad C_5>0,
	\]
	so we obtain (\ref{garding:1}) with $C_2=C_3+C_4C_5^2$. 
	\end{proof}
	
	We may now state the weighted version of the Morse index formula proved in \cite{delima1995local}.
	
	\begin{theorem}\label{morse:form:w}
		Let $\Sigma_1\subset \Sigma$ be such that
		$\Omega$ is $\Sigma_1$-critical with $\overline{\Sigma_1}$ compact. Then ${\rm ind}(\Sigma_1)<+\infty$. Moreover, if $\Sigma_0\subset \Sigma_1$ 
		 and there exists a smooth deformation $\Sigma_t\subset\Sigma$, $0\leq t\leq 1$, connecting $\Sigma_0$ to $\Sigma_1$
		 then there holds 
		 \begin{equation}\label{morse:form:w:2}
		 {\rm ind}(\Omega,\Sigma_1)-{\rm ind}(\Omega,\Sigma_0)=\sum_{0<t<1} {\rm null}(\Omega,\Sigma_t). 
		 \end{equation}
		 		\end{theorem}
		 		
	\begin{corollary}\label{shima:cor} \cite{shimakura1981premiere} If $\Omega$ is $\Sigma'$-critical and $x\in\Sigma'$ then there exits $U\subset\Sigma'$ with $x\in U$ such that $\Omega$ is strictly $U$-stable.		
		\end{corollary}
		
		\begin{proof} (of Theorem \ref{morse:form:w}) We resort to the abstract strategy in \cite{frid1990abstract}, so we only need to check the validity of two assertions regarding $\mathcal L_\phi$ and $Q_\phi$:
			\begin{enumerate}
				\item $Q_\phi$ satisfies a G\"arding inequality;
				\item $\mathcal L_\phi$ satisfies the {\em unique continuation property (UCP)} : if $\xi$ meets $\mathcal L_\phi\xi=\lambda \xi$, $\lambda\in\mathbb R$, and $\xi\equiv 0$ in some neighborhood $U\subset\Sigma_1$ then $\xi\equiv 0$ on $\Sigma$. 
			\end{enumerate}
			Now, the appropriate G\"arding inequality has already been established in Proposition \ref{garding}. As for UCP, notice that $\partial \xi/\partial \nu=$ on $U$. Since $\mathscr M_\phi=\Delta_\phi+\underline\lambda$ is elliptic, its principal symbol never vanishes and hence $U$ is {\em non-characteristic} for $\mathscr M_\phi$. Thus, by  H\"olmgren's uniqueness, $s_\xi$ vanishes in a neighborhood of $U$ (inside $\Omega$). But $\mathscr M_\phi$ is known to satisfy its own version of the UCP, which implies that $s_\xi\equiv 0$ and hence $\xi\equiv 0$  on $\Sigma$, as desired. This completes the proof. 
			\end{proof} 
			
			\begin{remark}\label{many:pot}
				 Among the many interesting (potential) applications of Theorem \ref{morse:form:w} we mention:
				\begin{enumerate}
					\item the annular domains in Remark \ref{cons:crit} are strictly $U$-stable if $U\subset \mathbb S^{n-1}_{r_i}(x_i)$ is small enough;
					\item any of the many critical domains constructed more recently by perturbative methods (see \cite{delay2013extremal} and the references therein) are locally strictly stable;
					\item very likely, the perturbative results mentioned in the previous item should admit counterparts in the weighted setting, so that local stability of critical domains should extend to this case as well;
					\item the nullity/index can now be identified to an algebraic count of negative/null eigenvalues of the associated quadratic form, which opens up the possibility of effective computations of these invariants in examples.  
				\end{enumerate} 
				\end{remark}

\section{The variatonal formulae for $J_{\beta,\gamma}$}\label{dich:energy}

We now turn to the Dirichlet energy $\mathcal E_{\beta,f}$ in Example \ref{ex:eq:lim}. We start by exhibiting an important example where the existence of a minimizer for the corresponding Dirichlet integral $J_{\beta,f}$ gets justified; cf. Remark \ref{suff:cond:inv}. 

\begin{example}\label{gauss:dir}
	In the Gaussian setting (\ref{gauss:def}), 
	the half-spaces $\mathcal H^n_{{\bf u},0}$ defined in (\ref{gauss:dir:g}) satisfy  
 $\underline\lambda(\mathcal H^n_{{\bf u},0})=1$ (with the corresponding eigenspace being generated by $u(x)=\langle x,{\bf u}\rangle_\delta$). Thus, 
	if $\Omega\varsubsetneq \mathcal H^n_{{\bf u},0}$ then $\underline\lambda(\Omega)> 1$ by eigenvalue  monotonicity, so that (\ref{dich:min:0}) has a solution $w$ for each $\beta\geq -1$ by Remark \ref{suff:cond:inv}, which means that $\mathcal E_{\beta,f}(\Omega)=J_{\beta,f}(w)$. As we shall  below, at least if $f\equiv\gamma\in\mathbb R$ the existence of this ``minimizer'' function $w$ will allow us to use variational methods to probe the nature of optimal domains for the problem of minimizing $\mathcal E_{\beta,f}(\Omega)$, where $\Omega\in \mathscr D_V$, $V={\rm vol}_{\phi_n}(\mathcal H^n_{{\bf u},d})$, $d>0$. 
	\end{example}	
	
As already observed, the general problem of characterizing optimal domains for $\mathcal E_{\beta,f}$, even in the rather special case of Example \ref{gauss:dir}, seem to lie beyond the current technology.	
Nevertheless, we may try to use the variational methods discussed above to obtain some insight on the properties an optimal domain should satisfy, at least in the case $f\equiv \gamma$, a real constant. As in Example \ref{gauss:dir} and taking Remark \ref{suff:cond:inv} into account, we will always assume below that $\beta>\underline{\lambda}(\Omega)$, so that $\mathcal E_{\beta,\gamma}(\Omega)=J_{\beta,\gamma}(w)$ with $w$ being the solution of (\ref{dich:min:0}) with $f=\gamma$. 
Granted this, it is a consequence of the first variation formula (\ref{f:der:en}) for $J_{\beta,\gamma}$ that optimal domains for $\mathcal E_{\beta,\gamma}$ satisfy the over-determined elliptic system (\ref{dich:min:o}) below. 

\begin{proposition}\label{opt:prop:f}
	If $\beta>\underline{\lambda}(\Omega)$ with $\Omega$ being optimal in the sense $\mathscr E_{V,\beta,\gamma}=\mathcal E_{\beta,\gamma}(\Omega)$ then the corresponding ``minimizer'' function $w$ (that is, $\mathcal E_{\beta,\gamma}(\Omega)=J_{c,\gamma}(w)$) satisfies the over-determined system
	\begin{equation}\label{dich:min:o}
		\left\{
		\begin{array}{lc}
			-\Delta_\phi w+\beta w=\gamma &  \Omega\\
			w=0  & \Sigma\\
			|\nabla w|= c & \Sigma
		\end{array}
		\right.
	\end{equation} 
	for some $c> 0$. 
	\end{proposition} 
	
	\begin{proof}
		If $V={\rm vol}_\phi(\Omega_0)$ and $t\in(-\epsilon,\epsilon)\mapsto \Omega_t\in \mathscr D_V$ is a smooth $\phi$-volume-preserving variation of $\Omega$ with $\Omega_0=\Omega$ then optimality of $\Omega$ together with the assumption $\beta>\underline{\lambda}(\Omega)$ imply that $J_{\beta,\gamma}(w)\leq J_{\beta,\gamma}(w_t)$, where $w_t$ is the solution of (\ref{dich:min:0}) with $\Omega=\Omega_t$ (and $f=\gamma$). Hence, $\dot{J}_{\beta,\gamma}=0$, where as always the dot means derivative at $t=0$. We now compute this derivative using the formalism stemming from Proposition \ref{main:formulae}. From (\ref{j:func:0}),
		\[
		J_{\beta,\gamma}(w_t)=\int_{\Omega_t}\Psi_t d{\rm vol}_\phi, \quad \Psi_t=\frac{1}{2}|\nabla w_t|^2+\frac{\beta}{2}w_t^2-\gamma w_t,
		\] 
		so that, by (\ref{main:formulae:b}),
		\begin{equation}\label{j:f:d}
			\dot{J}_{\beta,\gamma}
			 =  \int_{\Omega}\left(\langle \nabla w_0,\nabla \dot{w}\rangle+\beta w\dot{w}-\gamma\dot{w}\right)d{\rm vol}_\phi+\frac{1}{2}\int_\Sigma |\nabla w|^2\eta\, d{\rm area}_\eta, 
		\end{equation}
		where $\eta$ is the variational function and $\dot{w}$ satisfies 
		\begin{equation}\label{lin:eq}
			\left\{
		\begin{array}{lc}
			-\Delta_\phi \dot w+\beta\dot w=0 &  \Omega\\
			\dot w=-\frac{\partial w}{\partial \nu} \eta & \Sigma
		\end{array}
		\right.
		\end{equation}
		Now, by (\ref{int:part}), 
		\begin{eqnarray*}
			\int_{\Omega} \langle \nabla w,\nabla \dot{w}\rangle d{\rm vol}_\phi 
				& = & 
			-\int_\Omega \dot w\Delta_\phi w \,d{\rm vol}_\phi+\int_\Sigma \dot w\frac{\partial w_0}{\partial\nu}\,d{\rm area}_\phi \\
			& {=} &  
			-\int_\Omega \dot w(\beta w-\gamma) d{\rm vol}_\phi+
			\int_\Sigma \left(-\frac{\partial w}{\partial \nu}\eta\right)\frac{\partial w}{\partial \nu}d{\rm area}_\phi,
		\end{eqnarray*}
		and leading this to (\ref{j:f:d}) we see that 
		\begin{equation}\label{f:der:en}
		\dot J_{\beta,\gamma}=-\frac{1}{2}\int_\Sigma \left(\frac{\partial w}{\partial\nu}\right)^2\eta\, d{\rm area}_\phi.
		\end{equation}
		The result then follows (with $c=|\partial w/\partial \nu|$) because we already know that this derivative vanishes for any $\eta$ such that $\int_\Sigma\eta\, d{\rm area}_\phi=0$. 
	\end{proof}  
	
	As in Section \ref{var:form}, this proposition justifies a notion of $\Sigma$-criticality for $J_{\beta,\gamma}$, so that optimal domains for $\mathcal E_{\beta,\gamma}$ are automatically $\Sigma$-critical and hence satisfy (\ref{dich:min:o}). Naturally, the question remains of classifying such domains in each specific case. The inherent difficulty in approaching this problem is best illustrated by certain examples in Gaussian space taken from \cite[Section 2.3]{buttazzo2013some}.

	\begin{example}\label{butazzo}
		In Gaussian space of Example \ref{ex:gauss:pro}, the following domains are easily verified to be $\Sigma$-critical for $J_{-1,1}$ (equivalently, they support a function $w$ satisfying (\ref{dich:min:o}) for suitable values of $\beta$):
		\begin{enumerate}
			\item round balls centered at the origin;
			\item the complements of the balls in the previous item;
			\item the slabs $\{x\in\mathbb R^n;|\langle x,{\bf u}\rangle|<\varepsilon\}$, where $\varepsilon>0$ and $|{\bf u}|=1$; 
			\item the half-spaces $\mathcal H^n_{{\bf u},d}$ in (\ref{gauss:dir:g}); in this case the explicity solution to (\ref{dich:min:o}) for $d\geq 0$ is $w(x)=\beta x_1-1$ with $\beta d=1$.
		\end{enumerate}
		\end{example}
	
	Note the the set of all domains in each class of examples above exhausts the possible values for the $\phi_n$-volume of a proper domain in Gaussian space, which indicates that the problem of deciding which  $\Sigma$-critical domains are optimal is far from trivial given that at least four contenders compete for each value of the volume. As a first step toward approaching this difficulty we may look at how $J_{\beta,\gamma}$ varies to second order around such a domain. Thus, we proceed by computing $\ddot{J}_{\beta,\gamma}$ along variations passing through the given $\Sigma$-critical domain.      

We first observe that, as in Section \ref{var:form}, criticality here also implies that any $\eta\in H^{1/2}(\Sigma)$ admits a natural extension to $\Omega$. Indeed, it follows from (\ref{lin:eq}) that ${\mathfrak r}_\eta:=-(\partial w/\partial \nu)^{-1}\dot w$ satisfies 
	\begin{equation}\label{lin:eq:n}
	\left\{
	\begin{array}{lc}
		-\Delta_\phi {\mathfrak r}_\eta +\beta {\mathfrak r}_\eta=0 &  \Omega\\
		{\mathfrak r}_\eta= \eta & \Sigma
	\end{array}
	\right.
\end{equation}  	
which allows us to set $\partial \eta/\partial\nu:=\partial {\mathfrak r}_\eta/\partial\nu$ along $\Sigma$.

\begin{proposition}\label{sec:var:en}
	If $\Omega$ is $\Sigma$-critical for $J_{\beta,\gamma}$ with $\beta\neq 0$ then 
		\begin{equation}\label{sec:var:en:1}
		\ddot{J}_{\beta,\gamma}=c^2
		\mathscr Q_{\beta,\lambda}(\eta), \quad c=\left|\frac{\partial w}{\partial\nu}\right|,
	\end{equation}
	where the quadratic form $	\mathscr Q_{\beta,\gamma}$ is given by
	\begin{equation}\label{sec:var:en:2}
	\mathscr Q_{\beta,\gamma}(\eta)=\int_\Sigma \eta \mathcal L_{\beta,\gamma}\eta\,d{\rm area}_\phi, \quad \mathcal L_{\beta,\gamma}\eta=\frac{\partial\eta}{\partial\nu}+H_\phi\eta. 
	\end{equation}
	\end{proposition} 	
	
	\begin{proof}
		This uses essentially the same argument as in the previous computation of $\ddot{\underline\lambda}$ leading to (\ref{shim:q}). Indeed, 
	it follows from (\ref{f:der:en}) and (\ref{main:formulae:s}) that 
	\begin{equation}\label{var:gen:en}
		-{\ddot J}_{\beta,\gamma}=\frac{1}{2}\int_{\Sigma}\left(\dot\Psi+\left(\langle \nabla \Psi,\nu\rangle+\Psi H_\phi\right)\eta \right)d{\rm area}_{\phi}, \quad \Psi=\langle\nabla w_t,\nu_t\rangle^2\eta_t.	
	\end{equation}
	As before, we may assume that the variational vector field is normal along $\Sigma$, which immediately gives  
	\[
	\dot\Psi=-2c^2\xi\frac{\partial \xi}{\partial\nu}+c^2\dot\xi. 
	\]	
	On the other hand, 
	the only novelty in the computation of the remaining term inside the integral in (\ref{var:gen:en}) is that instead of (\ref{H:comp}) we now have 
	\[
	(\nabla^2 w)(\nu,\nu)=-\frac{\partial w}{\partial\nu}H_\phi-\gamma,
	\] 
	so we end up with 
	\begin{eqnarray*}
	{\ddot J}_{\beta,\gamma}
	& = & 
	c^2\int_{\Sigma} \eta\left(\frac{\partial \eta}{\partial\nu}+H_\phi\eta\right)d{\rm area}_{\phi}
	-\frac{\gamma}{2}\frac{\partial w}{\partial \nu}\int_{\Sigma_0}\eta\,d{\rm area}_{\phi}\\
	& & \quad 
	 -\frac{c^2}{2}\int_{\Sigma}\left(\dot\eta+\left(\frac{\partial\eta}{\partial\nu}+H_\phi\eta\right)\eta\right)d{\rm area}_{\phi},
	\end{eqnarray*}
	which reduces to (\ref{sec:var:en:2}) if we recall that the variation is $\phi$-volume-preserving and use (\ref{cond:mean}) and (\ref{cond:mean:2}) with $\xi=\eta$. 
	\end{proof}

We stress the formal similarity between (\ref{sec:var:en:1})-(\ref{sec:var:en:2}) and (\ref{shim:q})-(\ref{shim:qq}), the only essential difference being in the way the variational functions extend to $\Omega$ in each case. In particular, we can define here the notions corresponding  to those in Definitions \ref{stab} and \ref{index} above, so that the same argument as in the proof of Theorem \ref{morse:form:w} yields the following result. 

\begin{theorem}\label{morse:dirich}
	The appropriate Morse index formula, similar to (\ref{morse:form:w:2}), holds in the present setting. In particular, any $\Sigma'$-critical domain is locally strictly stable. 
\end{theorem}  

\begin{remark}\label{all}
As a consequence of Theorem \ref{morse:dirich}, all the $\Sigma$-critical domains (for $J_{-1,1}$) in Example \ref{butazzo} are locally strictly stable, thus being variationally indistinguishable from this viewpoint.  
\end{remark}

\section{A Pohozhaev identity and optimal domains for $\mathcal E_{-1,1}$ in Gaussian half-space}\label{class:opt}

As already observed in the Introduction and confirmed by Remark \ref{morse:dirich}, in general (infinitesimal) variational methods by themselves do not seem to provide effective tools for the classification of  
optimal domains for Dirichlet integrals. Thus, we henceforth seek for alternate routes by resorting to {\em global} methods relying on suitable differential/integral identities.    
We first illustrate this approach by means of the next result, which completely classifies (smooth) optimal domains for $\mathcal E_{-1,1}$ in the Gaussian half-space $\mathcal H^n_{{\bf u},0}$. It constitutes the exact analogue of a result in the spherical case \cite[Theorem 1.1]{ciraolo2019serrin} and we refer to Remark \ref{motiv:poin} for details on the heuristics behind this analogy, which relies on a naive application of Poincar\'e limit theorem (Appendix \ref{app:poinc}).

\begin{theorem}\label{class:th}
	If $\Omega\varsubsetneq \mathcal H^n_{{\bf u},0}$ is a $\Sigma$-critical domain for $J_{-1,1}$ in Gaussian space then $\Omega=\mathcal H^n_{{\bf u},d}$, $d>0$. 
	\end{theorem}
	
	\begin{corollary}\label{class:cor}
		If $\Omega\varsubsetneq \mathcal H^n_{{\bf u},0}$ is optimal for $\mathcal E_{-1,1}$ then $\Omega=\mathcal H^n_{{\bf u},d}$, $d>0$.
		\end{corollary}	

A key ingredient in our proof of Theorem \ref{class:th} is the following Pohozhaev-type identity in Gaussian space. 

\begin{proposition}\label{pohoz:ident}  
	If $w:\mathbb R^n\to\mathbb R$ is a $C^2$ function and $X$ is parallel (in the sense that $\nabla X\equiv 0$, where $\nabla$ is the covariant derivative induced by the flat metric $\delta$) then 
	\begin{equation}\label{pohoz:ident:2}
		{\rm div}_{\phi_n}\left(\frac{|\nabla w|^2}{2}X-\langle X,\nabla w\rangle\nabla w\right)=\frac{|\nabla w|^2}{2}{\rm div}_{\phi_n} X-\langle X,\nabla w\rangle\Delta_{\phi_n} w.
		\end{equation}
	\end{proposition}

\begin{proof}
	We have 
	\[
	{\rm div}_{\phi_n}\left(\frac{|\nabla w|^2}{2}X\right)=\frac{|\nabla w|^2}{2}{\rm div}_{\phi_n} X +
	\left\langle \nabla\left(\frac{|\nabla w|^2}{2}\right),X\right\rangle
	\]
	and 
	\[
		{\rm div}_{\phi_n}\left(\langle X,\nabla w\rangle\nabla w\right)=\langle X,\nabla w\rangle{\rm div}_{\phi_n}\nabla w+
		\langle\nabla\langle X,\nabla w\rangle,\nabla w\rangle, 
	\]
	and since 
	\[
	\left\langle \nabla\left(\frac{|\nabla w|^2}{2}\right),X\right\rangle=
		\langle\nabla\langle X,\nabla w\rangle,\nabla w\rangle
	\]
	because $X$ is parallel, the result follows. 
\end{proof}

We now observe that, by Proposition \ref{opt:prop:f}, any $\Omega$ satisfying the conditions of Theorem \ref{class:th} supports a function $w$ such that 
	\begin{equation}\label{dich:min:bvp}
	\left\{
	\begin{array}{lc}
		-\Delta_{\phi_n} w-w=1 &  \Omega\\
		w=0  & \Sigma \\
			|\nabla w| =c &  \Sigma	
	\end{array}
	\right.
\end{equation} 
for some $c>0$. This leads to our next preparatory result.

\begin{proposition}\label{weitz:conseq}
	If $w:\Omega\to\mathbb R$ satisfies (\ref{dich:min:bvp}) then 
	\begin{equation}\label{weitz:conseq:2}
	\Delta_{\phi_n}\left(|\nabla w|^2\right)\geq 0,
	\end{equation}
	with the equality occurring if and only if $\nabla^2w\equiv 0$. 
	\end{proposition}

\begin{proof}
	In the Gaussian case, the weighted Bochner-Weitzenb\"ock identity says that 
	\[
	\frac{1}{2}\Delta_{\phi_n}\left(|\nabla w|^2\right)=|\nabla^2w|^2+\langle \nabla w,\nabla(\Delta_{\phi_n} w)\rangle+|\nabla w|^2;
	\]
	see \cite[Appendix C.5]{bakry2014analysis}. 
	It follows that 
	\begin{eqnarray*}
		\frac{1}{2}\Delta_{\phi_n}\left(|\nabla w|^2\right)
		& \geq & \langle \nabla w,\nabla(-w-1)\rangle +|\nabla w|^2\\
		& = & 0, 
	\end{eqnarray*}
	with the equality occurring if and only if $|\nabla^2w|=0$.
\end{proof}

The final ingredient in the proof of Theorem \ref{class:th} is the next rigidity result for solutions of (\ref{dich:min:bvp}) in $\Omega\varsubsetneq \mathcal H^n_{{\bf u},0}$, in whose proof the Pohozhaev in Proposition \ref{pohoz:ident} plays a key role. 

\begin{proposition}\label{charact:fin}
	If $w$ satisfies (\ref{dich:min:bvp}) and $\Omega\varsubsetneq \mathcal H^n_{{\bf u},0}$ then $|\nabla w|=c$ everywhere on $\Omega$. 
	\end{proposition}
	
	\begin{proof}
		By rotational invariance we may assume that ${\bf u}={e_1}=(1,0,\cdots,0)$ so that $x_1\geq 0$ on $\Omega$. By (\ref{weitz:conseq:2}), 
		 the fact that $|\nabla w|=c$ along $\Sigma$ and 
		the appropriate maximum principle \cite[Theorem B]{bisterzo2024maximum}  we see that $|\nabla w|\leq c$ on $\Omega$, which immediately gives 
		\begin{equation}\label{ineq:w}
		c^2\int_\Omega x_1 d{\rm vol}_{\phi_n} > \int_\Omega x_1|\nabla w|^2 d{\rm vol}_{\phi_n} 
		\end{equation}
		unless there already holds $|\nabla w|=c$ everywhere on $\Omega$. 
		However, 
		\begin{eqnarray*}
		x_1|\nabla w|^2
		& = & 
		{\rm div}_{\phi_n} (x_1 w|\nabla w|)-w\langle \nabla x_1,\nabla w\rangle-x_1w\Delta_\phi w\\
		& = & {\rm div}_{\phi_n} (x_1 w|\nabla w|)-w\frac{\partial w}{\partial x_1}+x_1w(w+1),
		\end{eqnarray*}
		so if we take this to (\ref{ineq:w}) we obtain
		\begin{equation}\label{ineq:w:2}
		c^2\int_\Omega x_1 d{\rm vol}_{\phi_n} > 	
		-\int_\Omega w\frac{\partial w}{\partial x_1}d{\rm vol}_{\phi_n}+
		\int_\Omega x_1w(w+1)d{\rm vol}_{\phi_n}. 
			\end{equation}
		We now check that this contradicts the Pohozhaev-type identity (\ref{pohoz:ident:2}) for $X=-e_1$, for which there holds ${\rm div}_{\phi_n} X=x_1$. Indeed, integrating the left-hand side of (\ref{pohoz:ident:2}) over $\Omega$ with this choice and using the divergence theorem twice we get 
		\begin{eqnarray*}
		\int_\Sigma \left(\frac{c^2}{2}\langle X,\nu\rangle-\langle X,c\nu\rangle \langle c\nu,\nu\rangle\right)d{\rm area}_{\phi_n}
		& = & 
		-\frac{c^2}{2}\int_\Sigma \langle X,\nu\rangle d{\rm area}_{\phi_n}\\
		& = & -\frac{c^2}{2}\int_\Omega {\rm div}_{\phi_n} X d{\rm vol}_{\phi_n}\\
		& = & -\frac{c^2}{2}\int_\Omega x_1 d{\rm vol}_{\phi_n}.
		\end{eqnarray*}	
		On the other hand, the right-hand side of (\ref{pohoz:ident:2}) gives 
		\begin{eqnarray*}
		\frac{1}{2}\int_\Omega x_ 1|\nabla w|^2 d{\rm vol}_{\phi_n}-\int_\Omega \langle-e_1,\nabla w\rangle (-w-1)d{\rm vol}_{\phi_n}
		& = & 	\frac{1}{2}\int_\Omega x_ 1|\nabla w|^2 d{\rm vol}_{\phi_n}\\
		& & \quad -\int_\Omega \langle (w+1)e_1,\nabla w\rangle d{\rm vol}_{\phi_n}.
		\end{eqnarray*}
		Now, the last integral above may be manipulated as follows. We have
		\begin{eqnarray*}
		\langle (w+1)e_1,\nabla w\rangle 
		& = & {\rm div}_{\phi_n} (w(w+1)e_1)-w{\rm div}_{\phi_n}((w+1)e_1) 	\\
		& = & {\rm div}_{\phi_n} (w(w+1)e_1)-w(w+1){\rm div}_{\phi_n} e_1-w\langle \nabla w,e_1\rangle\\
		& = & {\rm div}_{\phi_n} (w(w+1)e_1)+x_1w(w+1)-w\frac{\partial w}{\partial x_1},
		\end{eqnarray*}
		so that 
		\[
		\int_\Omega \langle (w+1)e_1,\nabla w\rangle d{\rm vol}_{\phi_n}=\int_\Omega x_1w(w+1) d{\rm vol}_{\phi_n}-
		\int_\Omega w\frac{\partial w}{\partial x_1}  d{\rm vol}_{\phi_n},
		\]
		and we see 
		that altogether (\ref{pohoz:ident:2}) gives
		\[
		{c^2}\int_\Sigma x_1 d{\rm area}_{\phi_n}=
		-\int_\Omega x_ 1|\nabla w|^2 d{\rm vol}_{\phi_n}+2\int_\Omega x_1w(w+1) d{\rm vol}_{\phi_n}-2\int_\Omega w\frac{\partial w}{\partial x_1}  d{\rm vol}_{\phi_n}.
		\]
		Comparing this with (\ref{ineq:w:2}) we conclude that 
		\[
		\int_\Omega x_1w(w+1) d{\rm vol}_{\phi_n}-\int_\Omega w\frac{\partial w}{\partial x_1}  d{\rm vol}_{\phi_n}
		-\int_\Omega x_ 1|\nabla w|^2 d{\rm vol}_{\phi_n}>0.
		\]
		However, 
		\begin{eqnarray*}
			\int_\Omega x_1w(w+1) d{\rm vol}_{\phi_n}
			& = & -\int_\Omega x_1w\Delta_{\phi_n} w d{\rm vol}_{\phi_n}\\
			& = & \int_\Omega \langle\nabla (x_1w),\nabla w\rangle d{\rm vol}_{\phi_n}\\
			& = & \int_\Omega\left(w\frac{\partial w}{\partial x_1}+x_1|\nabla w|^2\right) d{\rm vol}_{\phi_n},
		\end{eqnarray*}
		so we end up with $0>0$, a contradiction which completes the proof. 
	\end{proof}
	
	\begin{proof} (of Theorem \ref{class:th}) By Proposition \ref{opt:prop:f} there exists $w:\overline\Omega\to\mathbb R$ satisfying (\ref{dich:min:bvp}), so Proposition \ref{charact:fin} applies to ensure that $|\nabla w|=c$ everywhere on $\Omega$. In particular, $\Delta_{\phi_n}(|\nabla w|^2)=0$ and the equality holds in (\ref{weitz:conseq}). Thus, by Proposition \ref{weitz:conseq}, $\nabla^2w\equiv 0$ and $w$ is an affine function. It follows that $\Sigma=\partial \Omega$ lies in a hyperplane (because $w$ vanishes there) and since 
	$\Omega\varsubsetneq \mathcal H^n_{{\bf u},0}$, the result follows.
	\end{proof}
	
	\begin{remark}\label{motiv:poin}
		Our proof of Theorem \ref{class:th} is modeled on \cite[Theorem 1.1]{ciraolo2019serrin}, where solutions of the over-determined system 
			\begin{equation}\label{dich:min:sph}
			\left\{
			\begin{array}{lc}
				-\Delta w-(k-1)Kw=1 &  \Omega\\
				w=0  & \Sigma \\
				|\nabla w| =c &  \Sigma	
			\end{array}
			\right.
		\end{equation}
		are shown to be quite rigid.
		Here, $\Omega\subset \mathscr S_{K,+}^{k-1}$, where $\mathscr S_{K,+}^{k-1}$ is a round hemisphere of dimensional $k-1$ and radius $1/\sqrt{K}$, $K>0$. In particular,  it is proved that $\Omega$ is a geodesic ball. Now, if we take $K=1/k$, so that $\mathscr S_{1/k,+}^{k-1}$ is a hemisphere of radius $\sqrt{k}$, and take the limit as $k\to +\infty$ then (\ref{dich:min:sph}) ``converges'' to the problem 
			\begin{equation}\label{dich:min:bvp:phi}
			\left\{
			\begin{array}{lc}
				-\Delta_\infty w-w=1 &  \Omega\\
				w=0  & \Sigma \\
				|\nabla w| =c &  \Sigma	
			\end{array}
			\right.
		\end{equation}
			Besides precisely determining the nature of the elusive ``Laplacian'' $\Delta_\infty$, the question remains of realizing in which space this limiting problem should be treated. From an entirely heuristic viewpoint, we may approach this by making use of the Poincar\'e limit theorem (see Appendix \ref{app:poinc} below), which may be loosely interpreted as saying that the push-forward of the weighted manifold $(\mathscr S_{1/k}^{k-1},\delta_{{\rm sph}_{k}},e^{-\phi_{(k)}}d{\rm vol}_{\delta_{{\rm sph}_{k}}})$ under the orthogonal projection $\Pi_{k,n}:\mathbb R^k\to\mathbb R^n$ converges in a suitable sense to the Gaussian space $(\mathbb R^n,\delta,e^{-\phi_n})$ as $k\to+\infty$. Here, $\delta_{{\rm sph}_{k}}$ is the standard spherical metric in $\mathscr S_{1/k}^{k-1}$ and $\phi_{(k)}>0$ is chosen such that ${\rm vol}_{\delta_{{\rm sph}_{k}}}(\mathscr S_{1/k}^{k-1})=e^{\phi_{(k)}}$, so that $\mathbb P_k:=e^{-\phi_{(k)}}d{\rm vol}_{\delta_{{\rm sph}_{k-1}}}$ defines a probability measure. This clearly suggests that, as we did above, the limiting problem (\ref{dich:min:bvp:phi}) should be treated in Gaussian half-space $\mathcal H^n_{{\bf u},0}=(\mathbb R^n_+,\delta,e^{-\phi_n}d{\rm vol}_\delta)$, with the {\em proviso} that  $\Delta_\infty$ must be replaced by the corresponding weighted Laplacian $\Delta_{\phi_n}$; see \cite{umemura1965infinite,mckean1973geometry} for a justification of this latter step. With the right over-determined system at hand, it remains to transplant to our setting the methods from \cite{ciraolo2019serrin}, which rely on 
			two key ingredients:
			\begin{itemize}
				\item Weinberger's approach \cite{weinberger1971remark} to Serrin rigidity is applied  to (\ref{dich:min:sph}) with the $P$-function
				\[
				P(w)=|\nabla w|^2+\frac{2}{k-1}w+K{w^2}\stackrel{k\to +\infty}{\longrightarrow} |\nabla w|^2,
				\]
				which naturally led to our much simpler choice in Proposition \ref{weitz:conseq};
				\item A Pohozhaev-type identity \cite[Proposition 3.1]{ciraolo2017rigidity} is required which, as it is well-known, turns out to be a manifestation of the existence of certain conformal fields on spheres which may be intrinsically characterized as gradient vector fields of functions lying in the first eigenspace of the metric Laplacian. Since spectral data behave quite well under Poincar\'e limit \cite[Theorem 1.1]{takatsu2022spectral}, we were naturally led to formulate Proposition \ref{pohoz:ident} in terms of {\em parallel} vector fields in Gaussian space, since they can be written as $\nabla \psi$ with $\psi$ satisfying $\Delta_{\phi_n}\psi+\psi=0$. Quite unsurprisingly, our Pohozhaev identity may be viewed as the Poincar\'e limit of theirs.
			\end{itemize}   
		\end{remark}

	\section{A Reilly formula and an Alexandrov-type theorem in Gaussian half-space}\label{alex:gauss}

	Inspired by Theorem \ref{class:cor}, and taking into account the heuristics behind it explained in Remark \ref{motiv:poin}, we are led to speculate on possible extensions of other classical results to the Gaussian setting by similar methods. The next theorem confirms this expectation and aligns itself with the well-known fact that Serrin's uniqueness result (or rather its reformulation by Weinberger \cite{weinberger1971remark}) may be used to provide an alternate proof of the classical Alexandrov's soap bubble theorem (cf. the comments in \cite{ciraolo2017rigidity}).     
	
	\begin{theorem}\label{alex:type:thm}
		Let $\Sigma\varsubsetneq \mathcal H^n_{{\bf u},0}$ be a smooth embedded hypersurface whose weighted mean curvature is (a non-zero) constant. Then $\Sigma=\partial \mathcal H^n_{{\bf u},d}$ for some $d> 0$. 
		\end{theorem}
		
		\begin{corollary}\label{alex:type:thm:c}
			There are no {\em compact} embedded hypersurfaces with (non-zero) constant weighted mean curvature in $ \mathcal H^n_{{\bf u},0}$.
		\end{corollary}

	We start the proof  of this result by recalling that if $B$ is a bi-linear symmetric form on vectors on $\Omega$, its $\phi$-{\em divergence} is given by
	\begin{equation}\label{div:def:B}
	{\rm div}_\phi B=e^\phi \circ {\rm div}\, B\circ e^{-\phi}={\rm div}\,B-{\bf i}_{\nabla \phi}B,
	\end{equation}
	where ${\rm div}\,B_j=\nabla^iB_{ij}$ is the metric divergence and  ${\bf i}_Y$ means  contraction by a vector field  $Y$. 
	
	\begin{proposition}\label{div:phi:B}
	Under the conditions above, if $h$ is a smooth function on $\overline\Omega$ then
	\begin{equation}\label{div:phi:B:2}
		{\rm div}_\phi(h {\bf i}_YB) 
		 =  B(\nabla h,Y) +h{\bf i}_Y{\rm div}_\phi B+\frac{h}{2}\langle B,\mathscr L_Yg\rangle,
	\end{equation}
	where $\mathscr L$ denotes Lie derivative. 
	\end{proposition}
	
	\begin{proof}
		We have
		\begin{eqnarray*}
			{\rm div}(h{\bf i}_YB)
			&= & 
			\nabla^i(hB_{ij}Y^j)\\
			& = & B_{ij}\nabla^i h Y^j+h\nabla^iB_{ij} Y^j+hB_{ij}\nabla^iY^j\\
			& = & B(\nabla h,Y) + h{\bf i}_Y{\rm div}\,B + 
			\frac{h}{2}\langle \mathscr L_Yg,B\rangle,
		\end{eqnarray*}
		and using (\ref{div:def:B}) twice,
		\begin{eqnarray*}
		{\rm div}_\phi(h{\bf i}_YB)
		& {=} &  
		B(\nabla h,Y) + h{\bf i}_Y{\rm div}B + 
		\frac{h}{2}\langle \mathscr L_Yg,B\rangle -h{\bf i}_{\nabla\phi}({\bf i}_YB)\rangle\\
		& = & 
			B(\nabla h,Y)+h{\bf i}_Y\left({\rm div}_\phi B+{\bf i}_{\nabla \phi}B\right)
			+ 
			\frac{h}{2}\langle \mathscr L_Yg,B\rangle -h{\bf i}_{\nabla\phi}({\bf i}_YB)\rangle,	
		\end{eqnarray*}
		so the result follows.
	\end{proof}
	
	It turns out that (\ref{div:phi:B:2}), for appropriate choices of $B$ and $Y$, is the starting point in establishing a Reilly-type formula holding for {\em any} weighted domain $(\Omega,g,d{\rm vol}_\phi)$ for which our standing assumptions hold.

	\begin{proposition}\label{reilly:w}
		Under our standing assumptions, if $h$ and $f$ are suitable functions on $\overline\Omega$ then there holds 
			\begin{eqnarray*}
			\int_\Omega h\left(\left(\Delta_\phi f+f\right)^2-|\nabla^2f|^2\right)
			& = &  \int_\Sigma h\left(2f_\nu\Delta_{\Sigma,\phi} f+H_\phi f_\nu^2+A_\Sigma(Df,Df)+2f_\nu f\right)\\
			& & 
			+\int_\Sigma h_\nu\left(|Df|^2-f^2\right)\\
			& & +\int_\Omega \left(
			(\nabla^2h)(\nabla f,\nabla f)-\Delta_\phi h\cdot |\nabla f|^2 -2h|\nabla f|^2+h{\rm Ric}_\phi(\nabla f,\nabla f)	\right)\\
			& & +\int_\Omega\left(\Delta_\phi h+h\right)f^2,
		\end{eqnarray*}
		where $A_\Sigma$ is the second fundamental form of $\Sigma$ and $f_\nu=\langle\nabla f,\nu\rangle$, etc.
		\end{proposition}

	\begin{proof}
	We take $B=\nabla^2 f$ and $Y=\nabla f$  in (\ref{div:phi:B:2}) and integrate over $\Omega$ to obtain 
	\[
	\int_\Omega h|\nabla^2f|^2=\int_\Sigma h(\nabla^2f)(\nabla f,\nu)-\int_\Omega (\nabla^2f)(\nabla h,\nabla f)-\int_\Omega h({\rm div}_\phi\nabla^2f)(\nabla f),
	\]
	where for convenience we omit the weighted area and volume elements in the integrals. 
	Now,
	\begin{eqnarray*}
	-\int_\Omega (\nabla^2f)(\nabla h,\nabla f) 
	& = & 
	 -\frac{1}{2}\int_\Omega\langle\nabla h,\nabla(|\nabla f|^2)\rangle \\
	 & = & 
	 \frac{1}{2}\int_\Omega \Delta_\phi h\cdot|\nabla f|^2-\frac{1}{2}\int_\Sigma |\nabla f|^2h_\nu
	\end{eqnarray*}
	and the Ricci identity easily implies that
	\[
	-\int_\Omega h({\rm div}_\phi\nabla^2f)(\nabla f)=-\int_\Omega h\langle\nabla(\Delta_\phi f),\nabla f\rangle -\int_\Omega h\,{\rm Ric}_\phi(\nabla f,\nabla f),
	\]
	so that 
	\begin{eqnarray*}
		\int_\Omega h|\nabla^2f|^2 
		& = &  \int_\Sigma h(\nabla^2f)(\nabla f,\nu) + \frac{1}{2}\int_\Omega \Delta_\phi h\cdot|\nabla f|^2-\frac{1}{2}\int_\Sigma |\nabla f|^2h_\nu\\
		& & -\int_\Omega h\langle\nabla(\Delta_\phi f),\nabla f\rangle -\int_\Omega h{\rm Ric}_\phi(\nabla f,\nabla f).
	\end{eqnarray*}
	On the other hand, the next to the last term in the right-hand side above may be treated as
	\begin{eqnarray*}
		-\int_\Omega h\langle\nabla(\Delta_\phi f),\nabla f\rangle
		& = & -\int_\Omega\langle \nabla(h\Delta_\phi f),\nabla f\rangle +\int_\Omega \Delta_\phi f\langle\nabla h,\nabla f\rangle \\
		& = & \int_\Omega  h(\Delta_\phi f)^2-\int_\Sigma h\Delta_\phi f\cdot f_\nu
		 +\int_\Omega \Delta_\phi f\langle\nabla h,\nabla f\rangle, 
	\end{eqnarray*}
	so we see that
	\begin{eqnarray*}
		\int_\Omega h|\nabla^2f|^2
	& = &  \int_\Sigma h(\nabla^2f)(\nabla f,\nu) -\frac{1}{2}\int_\Sigma |\nabla f|^2h_\nu
	+ \frac{1}{2}\int_\Omega \Delta_\phi h\cdot|\nabla f|^2\\
	& & -\int_\Sigma h\Delta_\phi f\cdot f_\nu +
	\int_\Omega  h(\Delta_\phi f)^2+\int_\Omega \Delta_\phi f\langle\nabla h,\nabla f\rangle\\
	& & -\int_\Omega h{\rm Ric}_\phi(\nabla f,\nabla f).	
	\end{eqnarray*}
	
	We now observe that 
	\begin{eqnarray*}
	\int_\Omega h\left(\left(\Delta_\phi f+f\right)^2-|\nabla^2f|^2\right)
	& = & \int_\Omega h(\Delta_\phi f)^2+ 2\int_\Omega h\Delta_\phi f\cdot f \\
	&  & +\int_\Omega hf^2-\int_\Omega h|\nabla^2 f|^2
	\end{eqnarray*}
	and 
	\begin{eqnarray*}
		\int_\Omega h\Delta_\phi f\cdot f
		& = & \int_\Sigma h ff_\nu-\int_\Omega \langle\nabla(hf),\nabla f\rangle\\
		& = & \int_\Sigma h ff_\nu-\int_\Omega h|\nabla f|^2 -\int_\Omega f\langle \nabla h,\nabla f\rangle, 
	\end{eqnarray*}
	so we obtain 
	\begin{eqnarray*}
	\int_\Omega h\left(\left(\Delta_\phi f+f\right)^2-|\nabla^2f|^2\right)
	& = & \int_\Omega h(\Delta_\phi f)^2 +
	2\int_\Sigma h ff_\nu-2\int_\Omega h|\nabla f|^2 -2\int_\Omega f\langle \nabla h,\nabla f\rangle\\
	& &  +\int_\Omega h f^2 \\
	& &  -\int_\Sigma h(\nabla^2f)(\nabla f,\nu) +\frac{1}{2}\int_\Sigma |\nabla f|^2h_\nu
	- \frac{1}{2}\int_\Omega \Delta_\phi h\cdot|\nabla f|^2\\
	& & +\int_\Sigma h\Delta_\phi f\cdot f_\nu -
	\int_\Omega  h(\Delta_\phi f)^2-\int_\Omega \Delta_\phi f\langle\nabla h,\nabla f\rangle\\
	& & +\int_\Omega h{\rm Ric}_\phi(\nabla f,\nabla f)\\
	& = & 
	\int_\Sigma h\Delta_\phi f\cdot f_\nu +\frac{1}{2}\int_\Sigma |\nabla f|^2h_\nu
	-\int_\Sigma h(\nabla^2f)(\nabla f,\nu) +	2\int_\Sigma h ff_\nu\\
	&  & - \frac{1}{2}\int_\Omega \Delta_\phi h\cdot|\nabla f|^2 
	-\int_\Omega \Delta_\phi f\langle\nabla h,\nabla f\rangle  +\int_\Omega h\,{\rm Ric}_\phi(\nabla f,\nabla f)\\
	& & -2\int_\Omega h|\nabla f|^2 -2\int_\Omega f\langle \nabla h,\nabla f\rangle +\int_\Omega h f^2.
	\end{eqnarray*}
	Since we can handle the sixth and ninth terms in the right-hand side above as 
	\begin{eqnarray*}
		-\int_\Omega \Delta_\phi f\langle\nabla h,\nabla f\rangle
		& = & -\int_\Sigma f_\nu\langle\nabla h,\nabla f\rangle  + \int_\Omega\langle\nabla f,\nabla\langle\nabla h,\nabla f\rangle\rangle\\
		& = & -\int_\Sigma f_\nu\langle\nabla h,\nabla f\rangle +\int_\Omega(\nabla^2h)(\nabla f,\nabla f)+\frac{1}{2}\int_\Omega \langle \nabla h,\nabla |\nabla f|^2\rangle\\
		& = & -\int_\Sigma f_\nu\langle\nabla h,\nabla f\rangle +\int_\Omega(\nabla^2h)(\nabla f,\nabla f)\\
		&  & +\frac{1}{2}\int_\Sigma |\nabla f|^2h_\nu-\frac{1}{2}\int_\Omega \Delta_\phi h\cdot |\nabla f|^2
	\end{eqnarray*}
	and 
	\begin{eqnarray*}
		-2\int_\Omega f\langle \nabla h,\nabla f\rangle 
		& = & -\int_\Omega \langle \nabla h,\nabla f^2\rangle\\
		& = & -\int_\Sigma f^2 h_\nu +\int_\Omega f^2\Delta_\phi h,
	\end{eqnarray*}
	we end up with 
	\begin{eqnarray}\label{in:term}
		\int_\Omega h\left(\left(\Delta_\phi f+f\right)^2-|\nabla^2f|^2\right)
		& = & 	\int_\Sigma h\left(\Delta_\phi f \cdot f_\nu-(\nabla^2f)(\nabla f,\nu)\right) \nonumber\\
		& & +\int_\Sigma\left(|\nabla f|^2h_\nu-f_\nu\langle\nabla h,\nabla f\rangle\right)\nonumber\\
		& & + 2\int_\Sigma hff_\nu -\int_\Sigma h_\nu f^2\nonumber\\
		& & +\int_\Omega \left(
		(\nabla^2h)(\nabla f,\nabla f)-\Delta_\phi h\cdot |\nabla f|^2 -2h|\nabla f|^2+h{\rm Ric}_\phi(\nabla f,\nabla f)	\right)\nonumber\\
	& & +\int_\Omega\left(\Delta_\phi h+h\right)f^2.
	\end{eqnarray}
	
	We now handle the first two boundary terms above by computing in an adapted frame along $\Sigma$. First,
	\begin{eqnarray*}
	\int_\Sigma h\left(\Delta_\phi f \cdot f_\nu-(\nabla^2f)(\nabla f,\nu)\right)
	& = & -\int_\Sigma h\langle \nabla \phi,\nabla f\rangle	f_\nu + 
	\int_\Sigma h\left(\Delta f \cdot f_\nu-(\nabla^2f)(\nabla f,\nu)\right)\\
	& = & -\int_\Sigma h\langle \nabla \phi,\nabla f\rangle	f_\nu \\
	& & +\int_\Sigma h\left(f_\nu\Delta_{\Sigma} f+H f_\nu^2-\langle D f_\nu,Df\rangle+A_\Sigma(Df,Df)\right)\\
	& = & -\int_\Sigma h\langle \nabla \phi,\nabla f-f_\nu\nu\rangle	f_\nu
	 +\int_\Sigma hf_\nu \langle D\phi,Df\rangle\\
	&  &  +\int_\Sigma h\left(f_\nu\Delta_{\Sigma,\phi} f+H_\phi f_\nu^2-\langle D f_\nu,Df\rangle+A_\Sigma(Df,Df)\right)\\
	& = &  \int_\Sigma h\left(f_\nu\Delta_{\Sigma,\phi} f+H_\phi f_\nu^2
	-\langle D f_\nu,Df\rangle+A_\Sigma(Df,Df)\right),
	\end{eqnarray*}
	where we used Gauss-Weingarten in the second step.
	Also, 
	\begin{eqnarray*}
		\int_\Sigma\left(|\nabla f|^2h_\nu-f_\nu\langle\nabla h,\nabla f\rangle\right)
		& = & \int_\Sigma\left(|Df|^2 h_\nu-f_\nu\langle Dh,Df\rangle\right)\\
		& = & \int_\Sigma\left(|D f|^2h_\nu -\langle D(f_\nu h),Df\rangle
		+h\langle Df_\nu,Df\rangle\right)\\
		& = & \int_\Sigma\left(
		|D f|^2h_\nu +h\langle Df_\nu,Df\rangle
		+f_\nu h\Delta_{\Sigma,\phi}f
		\right),
	\end{eqnarray*}
	so if we substitute these identities back in (\ref{in:term}) the result follows after a few manipulations. 
	\end{proof}
	
	\begin{proof} (of Theorem \ref{alex:type:thm}) By rotational invariance we may assume that ${\bf u}=e_1$. We apply our Reilly formula in Proposition \ref{reilly:w} with $h=x_1$, so that $\Delta_{\phi_n} h+h=0$ and choose $f$ satisfying the Dirichlet problem
	\begin{equation}\label{dir:f}
		\left\{
		\begin{array}{lc}
			\Delta_{\phi_n} f+f=1 &  \Omega\\
			f=0  & \Sigma 
		\end{array}
		\right.
	\end{equation}
	where we may assume that $\Sigma\neq \partial \mathcal H^n_{e_1,0}$, hence bounding a domain $\Omega\varsubsetneq \mathcal H^n_{e_1,0}$; cf. Remark \ref{suff:cond:inv}. 
	Since ${\rm Ric}_{\phi_n}=\delta$, the formula reduces to 
	\[
	\int_\Omega x_1-\int_\Omega x_1|\nabla^2f|^2=\int_\Sigma H_{\phi_n} x_1f_\nu^2= H_{\phi_n}\int_\Sigma x_1f_\nu^2,
	\]	
	and using Cauchy-Schwarz, 
	\begin{equation}\label{c:s}
			\int_\Omega x_1-\int_\Omega x_1|\nabla^2f|^2\geq H_{\phi_n}\frac{\left(\int_\Sigma x_1 f_\nu\right)^2}{\int_\Sigma x_1}. 
	\end{equation}
	Now, as a consequence of (\ref{int:part}),
	\[
	\int_\Sigma x_1 f_\nu =\int_\Omega x_1\Delta_{\phi_n} f - \int_\Omega f\Delta_{\phi_n} x_1=\int_\Omega x_1.
	\]
	On the other hand, the Minkowski formula in Gaussian space \cite[Corollary 3.4]{abdelmalek2021generalized} gives
	\[
	\int_\Sigma x_1=-H_{\phi_n}\int_\Sigma\langle e_1,\nu\rangle,
	\]
	and since 
	\[
	\int_\Sigma\langle e_1,\nu\rangle=\int_\Omega {\rm div}_{\phi_n} e_1=-\int_\Omega x_1, 
	\]
	we see from (\ref{c:s}) that $\nabla^2f\equiv 0$ and therefore $f$ is an affine function. Thus, 
	$\Sigma$ lies in a hyperplane and since $\Omega\subset \mathcal H^n_{e_1,0}$, the proof is completed. 
		\end{proof}
	
	\begin{remark}\label{opt:ppp}
	Proposition \ref{opt:prop:f} implies that an optimal domain $\Omega\varsubsetneq \mathcal H^n_{e_1,0}$ for $\mathcal E_{-1,-1}$ supports a function $f$ satisfying the over-determined sustem
	\begin{equation}\label{dir:f:2}
		\left\{
		\begin{array}{lc}
			\Delta_{\phi_n} f+f=1 &  \Omega\\
			f=0  & \Sigma \\
			|\nabla f|=c & \Sigma
		\end{array}
		\right.
	\end{equation}
	for some $c>0$. 
The proof above  
	shows that if we replace the third condition in (\ref{dir:f:2}) by the assumption that the weighted mean curvature of $\Sigma=\partial \Omega$ is constant then rigidity is recovered.
		\end{remark}
	
	\begin{remark}\label{mod:alex}
		Our argument leading to Theorem \ref{alex:type:thm} is modeled on the proof of \cite[Theorem 1.2]{qiu2015generalization}, where Alexandrov's soap bubble theorem in simply connected space forms is retrieved as a consequence of a generalized Reilly formula \cite[Theorem 1.1]{qiu2015generalization}. It turns out that if we apply their formula to (a compact domain in) the round hemisphere $\mathscr S^{k-1}_{1/k,+}$ in Remark \ref{motiv:poin}, take the limit as $k\to +\infty$ and argue (always on heuristic grounds!) that not only the metric Laplacian but also the standard mean and Ricci curvatures should give rise to their weighted counterparts along Poincar\'e convergence, we obtain our Reilly-type formula in Proposition \ref{reilly:w} (as applied to Gaussian space). Of course, the long computation in the proof of this latter proposition provides a rigorous justification of this rather informal argument based on Poincar\'e limit (Appendix  \ref{app:poinc}). Finally, we mention that Proposition \ref{reilly:w} (again, as applied to Gaussian space) also follows formally as the appropriate Poincar\'e limit of a Reilly-type formula in \cite[Theorem 1.2]{deng2020sharp}, which by its turn generalizes to weighted domains the one in \cite{qiu2015generalization} referred to above.
		\end{remark}
		
		\begin{remark}\label{cheng:sur}
			Hypersurfaces in weighted manifolds with constant weighted mean curvature are usually referred to as $\lambda$-{\em hypersurfaces} and there is a huge literature on trying to classify them under a most varied set of assumptions. We refer to \cite{cheng2022complete} for a recent survey on the subject. 
			\end{remark}

\appendix
\section{Poincar\'e limit theorem}\label{app:poinc}

	For the reader's convenience, we provide here the precise statement of Poincar\'e limit theorem \cite{hida1964gaussian,umemura1965infinite,mckean1973geometry,diaconis1987dozen}, which has been used in a rather informal way in the bulk of the text,  and then complement it with a brief discussion on its modern formulation in the framework of Gromov's metric measure geometry \cite{gromov1999metric,shioya2016metric,shioya2022metric}, in which it appears as an instance of ``phase transition'' between the extreme regimes arising as the limits of suitably re-scaled spheres. Taken together, these formal results provide appreciable support to the heuristics described in Remarks \ref{motiv:poin} and \ref{mod:alex}, which  led to Theorems \ref{class:th} and \ref{alex:type:thm}, respectively.  
	
	Using the notation of Remark \ref{motiv:poin}, let  
	$\Pi_{k,n}:\mathbb R^k\to\mathbb R^n$ be the orthogonal projection associated to the natural embedding $\mathbb R^n\hookrightarrow\mathbb R^k$ and let $(\mathscr S^{k-1}_{1/k},\mathbb P_{k})$ be the 
	probability space defined by $\mathbb P_k=e^{-\phi_{(k)}}d{\rm vol}_{\delta_{{\rm sph}_k}}$.
 With this terminology, and using standard probabilistic jargon, Poincar\'e limit theorem says that the  
	random vectors 
	\[
	\widetilde{\Pi}_{k,n}=\Pi_{k,n}|_{\mathscr S^{k-1}_{1/k}}:(\mathscr S^{k-1}_{1/k},\mathbb P_k)\to\mathbb R^n
	\]
	 converge {\em weakly} to a $\mathbb R^n$-valued random vector $Z$ whose distribution is the Gaussian density $d{\rm vol}_{\phi_n}=e^{-\phi_n}d{\rm vol}_\delta$,
	which means that $\mathbb E(\psi(\widetilde{\Pi}_{k,n}))\to \mathbb E(\psi(Z)))$ as $k\to+\infty$ for any $\psi:\mathbb R^n\to \mathbb R$ uniformly bounded and continuous. The standard probabilistic proof of this remarkable result combines the Law of Large Numbers and the well-known fact that a 
	random vector $Z^{[k]}$ uniformly distributed over $\mathscr S^{k-1}_{1/k}$ may be expressed as 
	\[
		Z^{[k]}=\sqrt{k} \frac{X^{[k]}}{\|X^{[k]}\|},
	\]
	where $X^{[k]}$ is a $\mathbb R^k$-valued standard Gaussian vector; see \cite[Proposition 6.1]{lifshits2012lectures} for a proof along these lines. 
	With a bit more of effort it may be checked that 
	this statement actually holds true with $\psi={\bf 1}_A$, the indicator function 
	of an arbitrary Borel set $A\subset \mathbb R^n$. Precisely, 
	\begin{equation}\label{poinc:sets}
		\lim_{k\to +\infty}
		\mathbb P_{k}(\widetilde\Pi_{k,n}^{-1}(A))
		={\rm vol}_{\phi_n}(A); 
	\end{equation}
	see \cite[Section 1.1]{ledoux2013probability} or \cite[Section 1]{ledoux2006isoperimetry}, where it is also explained how this sharpened version may be used to transfer the solution of the isoperimetric problem from $(\mathscr S^{k-1}_{1/k},\mathbb P_k)$ to Gaussian space as $k\to +\infty$, a celebrated result first proved by Borell \cite{borell1975brunn} and Sudakov-Tsirel'son \cite{sudakov1978extremal} and which lies at the heart of the ``concentration of measure phenomenon'' \cite{gromov1999metric,ledoux2001concentration,shioya2016metric}. 
	Inspired by this circle of ideas we are led to loosely interpret (\ref{poinc:sets})  as saying that the ``projections'' of the weighted manifolds $(\mathscr S^{k-1}_{1/k},\delta_{{\rm sph}_k},\mathbb P_k)$ onto $\mathbb R^k$ converge (in a sense whose actual meaning is not relevant at this point) to the Gaussian space $(\mathbb R^n,\delta,d{\rm vol}_{\phi_n})$, a perspective we have naively explored above in order to transplant problems and techniques from (high dimensional) spheres to Gaussian space. 
	We refer to 
	\cite[Remarks 5.14 and 5.15]{delima2023probab} 
	for a detailed guide to the proofs of both versions of Poincar\'e limit theorem mentioned above.    
	
Regarding this theory, a major development occurred when Gromov \cite[Chapter 3$\frac{1}{2}$]{gromov1999metric} outlined the construction of an explicit compactification, say $\Xi$, for the collection $\mathcal X$ of all weighted spaces (endowed with the {\em observable distance}, its natural metric, and with no restriction whatsoever on the dimension and/or diameter). In fact, Gromov considers the more general class of {\em mm-spaces} formed by complete separable metric spaces $(X,d)$ endowed with a probabilty measure $\mu$ on their Borel algebra, so that in particular closed Riemannian manifolds are always equipped with their normalized volume elements. Roughly speaking, 
the underlying idea is that in general 
the limit of a sequence of mm-spaces is no longer an mm-space but rather a sort of ideal object called a {\em pyramid}, so that in
order to perform the topological completion one has to attach to $\mathcal X$ these ideal configurations. 
With this provision, the compactifying embedding map $\mathcal P:\mathcal X\hookrightarrow \Xi$ associates to each metric measure space $(X,d,\mu)$ its pyramid, say $\mathcal P_{X}$. 
In this setting, the important notion of $\kappa$-{\em observable diameter}, $0<\kappa<1$, allows for a modern rewriting of an old computation due to P. L\'evy \cite{levy1951problemes}:
\begin{equation}\label{levy}
\kappa{\mbox{-}}{\rm observable}\,{\rm diameter}\left(\mathcal P_{\mathscr S^{k-1}_{1}}\right)\sim_{\kappa\mkern-9.5mu/}\frac{1}{\sqrt{k}}, 
\end{equation}
where $a_k(\kappa)\sim_{\kappa\mkern-9.5mu/} b_k$ means that the sequences $a_k(\kappa)/b_k$ and $b_k/a_k(\kappa)$ are both uniformly bounded as $k\to +\infty$ {\em independently} of $\kappa$.
More recently, Gromov's construction has been revisited by Shioya and collaborators; see \cite{shioya2016metric,shioya2022metric} and the references therein. 
In particular, a metric on $\Xi$ has been exhibited which is compatible with the topology of the compactification. 
In this reassessment, (\ref{levy}) is viewed as the fulfillment of a criterion for the validity of a phase transition property guaranteeing that suitable re-scalings of $\mathcal P_{\mathscr S^{k-1}_{1}}$ converge to the extreme pyramids $\mathcal P_{\bullet}$ and $\mathcal P_{\mathcal X}$, according to the case \cite[Theorem 1.2]{ozawa2015limit}; here, $\bullet$ is the one-point mm-space and $\mathcal P_{\mathcal X}$ is the pyramid associated to the whole colection $\mathcal X$ of mm-spaces. 
This remarkable dichotomy is 
then sharpened by means of the following striking result, which identifies 
$\sqrt{k}$ as being the {\em critical scale order} of the underlying phase transition and in whose proof Poincar\'e limit plays a key role. 

\begin{theorem}\label{phase}(Poincar\'e limit as a phase transition \cite[Theorem 8.1.1]{shioya2017metric})
	For any given sequence $\{r_k\}$ of positive real numbers there hold:
	\begin{enumerate}
		\item $r_k/\sqrt{k}\to 0$ if and only if $\mathcal P_{\mathscr S^{k-1}_{1/r_k^2}}\to \mathcal P_{\bullet}$;
			\item $r_k/\sqrt{k}\to +\infty$ if and only if $\mathcal P_{\mathscr S^{k-1}_{1/r_k^2}}\to \mathcal P_{\mathcal X}$;
		\item $r_k/\sqrt{k}\to \sigma$, $0<\sigma<+\infty$, 	if and only if $\mathcal P_{\mathscr S^{k-1}_{1/r_k^2}}\to \mathcal P_{\Gamma^\infty_{\sigma^2}}$, where $\Gamma^\infty_{\sigma^2}$ is the infinite dimensional Gaussian space of variance $\sigma^2$.
	\end{enumerate}
	\end{theorem}

We remark that (1) above, which portrays $\mathscr S^{k-1}_{1/r_k^2}$ as a {\em L\'evy family}, already appears in \cite[Section 1]{gromov1983topological}. This is of course the most relevant case in applications such as concentration inequalities \cite{ledoux2001concentration}, essentially because in this setting observable diameters roughly correspond to (normal) standard deviations. 
However, if one intends to avoid the extreme cases displayed in (1) and (2) above, so as to hope for the emergence of some sort of geometry in the limit, one is forced to re-scale  $\mathcal P_{\mathscr S^{k-1}_{1}}$ by  $r_k\sim\sqrt{k}$, as in Poincar\'e limit discussed earlier. It is clear, 
moreover, that the third item (with $\sigma=1$) adds considerable support to the heuristics that guided us in formulating and proving the rigidity results in Theorems \ref{class:th} and \ref{alex:type:thm}.  

\bibliographystyle{alpha}
\bibliography{faber}

\end{document}